\documentclass{article}
\usepackage{graphicx} 

\title{Exact Recovery in the Geometric Hidden Community Model}
\author{Julia Gaudio and Andrew Jin}

\usepackage[margin=1in]{geometry}
\usepackage{graphicx} 
\usepackage{amsfonts}
\usepackage{enumitem}
\usepackage{amsmath}
\usepackage{amssymb}
\usepackage{amsthm}
\usepackage{mathtools}
\usepackage{comment}
\usepackage{hyperref}
\usepackage[dvipsnames]{xcolor}
\usepackage{algorithm}
\usepackage[noend]{algpseudocode}
\usepackage{bbm}
\usepackage{placeins}

\DeclareMathOperator*{\argmin}{argmin}
\DeclareMathOperator*{\argmax}{argmax}

\newcommand{\R}{\mathbb{R}}
\renewcommand{\P}{\mathbb{P}}
\newcommand{\E}{\mathbb{E}}
\newcommand{\N}{\mathbb{N}}
\newcommand{\Z}{\mathbb{Z}}
\newcommand{\1}{\mathbbm{1}}
\newcommand{\calS}{\mathcal{S}}
\newcommand{\calN}{\mathcal{N}}
\newcommand{\calE}{\mathcal{E}}
\newcommand{\calA}{\mathcal{A}}
\newcommand{\calB}{\mathcal{B}}
\newcommand{\calH}{\mathcal{H}}
\newcommand{\calI}{\mathcal{I}}
\newcommand{\calJ}{\mathcal{J}}
\newcommand{\calL}{\mathcal{L}}

\newcommand{\calX}{\mathcal{X}}
\newcommand{\calY}{\mathcal{Y}}

\newcommand{\visible}{\leftrightarrow{}}

\theoremstyle{definition}
\newtheorem{definition}{Definition}
\newtheorem{assumption}{Assumption}

\theoremstyle{plain}
\newtheorem{theorem}{Theorem}
\newtheorem{lemma}{Lemma}
\newtheorem{corollary}{Corollary}
\newtheorem{fact}{Fact}
\newtheorem{prop}{Proposition}

\begin{document}

\maketitle

\begin{abstract}
Hidden community problems, such as community detection in the Stochastic Block Model (SBM), submatrix localization, and $\Z_2$ synchronization, have received considerable attention in the probability, statistics, and information-theory literature. Motivated by transitive behavior in social networks, which tend to exhibit high triangle density, recent works have considered hidden community models in spatially-embedded networks. In particular, Baccelli and Sankararaman proposed the Geometric SBM, a spatially-embedded analogue of the standard SBM with dramatically more triangles. In this paper, we consider the problem of exact recovery for the Geometric Hidden Community Model (GHCM) of Gaudio, Guan, Niu, and Wei, which generalizes the Geometric SBM to allow for arbitrary pairwise observation distributions. Under mild technical assumptions, we find the information-theoretic threshold for exact recovery in the ``distance-dependent'' GHCM, which allows the pairwise distributions to depend on distance as well as community labels, thus completing the picture of exact recovery in spatially-embedded hidden community models. 
\end{abstract}

\section{Introduction}
\label{sec:introduction}
Community detection is the problem of identifying hidden dense communities in networks. The Stochastic Block Model (SBM), introduced by Holland, Laskey, and Leinhardt in 1983 \cite{Holland1983}, is considered the canonical model for generating random graphs with community structure. In its simplest version, the SBM comprises two communities of vertices of equal size. Pairs of vertices in the same community are connected with probability $p$, while pairs of vertices in opposite communities are connected with probability $q$. If $p > q$, then the graph exhibits assortative community structure; that is, typical instances of the SBM have two dense communities with sparse connections between them. Depending on the values of $p$ and $q$, it may be possible to recover the communities exactly, almost exactly, or partially. These inference tasks have been the subject of a large body of research, much of it summarized in the survey of Abbe \cite{Abbe2018}. 

An emerging line of work studies spatially-embedded versions of community detection problems, which lead to \emph{transitive} community structure; namely, if $(u,v)$ and $(u,w)$ are edges, then $(v,w)$ is likely to be an edge as well. Recently, Baccelli and Sankararaman proposed the Geometric Stochastic Block Model (GSBM) \cite{Sankararaman2018}, a spatially-embedded analogue of the SBM in which vertices are placed in a bounded region of $\mathbb{R}^d$ according to a homogeneous Poisson process, and vertices are connected by edges with probabilities that depend on both their communities as well as their distance. An important special case of the GSBM is one where vertices may only be connected by an edge if they are sufficiently close to each other (i.e., if the distance between $u,v$ satisfies $\Vert u - v \Vert \leq r_n$); in that case, they are connected with probability $a$ if they are in the same community and with probability $b$ otherwise. We refer to this as the \emph{constant-probability} GSBM. In this special case, the GSBM has dramatically more triangles than the standard SBM. Fixing the average degree to be the same in the two models, the GSBM with $\text{Pois}(\lambda n)$ vertices has order $n$ times more triangles than the SBM with $n$ vertices. 

Following the introduction of the GSBM, subsequent works focused on the exact recovery problem in the constant-probability GSBM. Gaudio, Niu, and Wei \cite{Gaudio+2024} determined the exact recovery threshold in the logarithmic-degree regime of the constant-probability GSBM, complementing the impossibility result that had been established by Abbe et al \cite{Abbe+2020}. Furthermore, \cite{Gaudio+2024} proposed a linear-time algorithm for exact recovery which achieves the information-theoretic (IT) threshold. Gaudio, Guan, Niu, and Wei \cite{Gaudio+2025} then proposed the Geometric Hidden Community Model (GHCM), a generalization of the GSBM which allows for multiple, possibly unbalanced communities, and considers pairwise distributions that are not necessarily Bernoulli-valued (i.e. ``edge'' or ``no edge''). The GHCM thus extends prominent hidden community models like submatrix localization and $\mathbb{Z}_2$ synchronization to the spatially-embedded setting, capturing settings where pairwise observations are available only for vertices which are ``close'' in terms of their physical distance or their feature embedding. For example, consider the problem of sensor calibration in a physical network, in which relative measurements are available only for pairs of sensors that are close to each other. A variation of the algorithm of \cite{Gaudio+2024} was shown to achieve the IT threshold in the constant-probability GHCM, under some technical assumptions. Subsequently, Gaudio and Guan \cite{GaudioGuan2025} showed that one of these assumptions, which requires a certain ``distinctness of distributions,'' is not required in the two-community setting, thus allowing for a spatially-embedded analogue of the planted dense subgraph and submatrix localization problems, among other hidden community problems. 

Moving away from the constant-probability assumption, Avrachenkov, Kumar, and Leskel\"a \cite{Avrachenkov2024} considered the setting of two communities, where edge probabilities are governed by a pair of functions $f_{\text{in}}(\cdot)$ and $f_{\text{out}}(\cdot)$: if $u,v$ are in the same community, they are connected with probability $f_{\text{in}}(\Vert u - v\Vert)$; otherwise, they are connected with probability $f_{\text{out}}(\Vert u - v\Vert)$. The authors considered only the case where $d=1$ and the functions differ by a constant factor. Subsequently, Gaudio and Jin \cite{GaudioJin2025} handled the case of arbitrary, constant dimension, where $f_{\text{in}}(\cdot)$ and $f_{\text{out}}(\cdot)$ only need to satisfy some regularity conditions but otherwise can be quite general.

In this paper, we complete the picture for the exact recovery problem in the GHCM by allowing for distance dependence in pairwise distributions. We determine the IT threshold for exact recovery and show that an extension of the algorithms proposed by \cite{Gaudio+2025} and \cite{GaudioJin2025} achieves the IT threshold. Intuitively, the IT threshold is determined by the probability that a vertex can be accurately labeled when all other labels are accurately given. Since there are $\Theta(n)$ vertices, the IT threshold occurs when the probability of misclassifying a given vertex is $\Theta(1/n)$. That is, if the error probability is $\omega(1/n)$, exact recovery is impossible, while if it is $o(1/n)$, then exact recovery is possible. 

While these asymptotics are intuitive, they are not immediate. The asymptotics of the impossibility result are established in \cite{Abbe+2020}, and lower-bounding the error probability is accomplished by employing a Cram\'er lower bound from large deviations theory, following the analysis of \cite{Gaudio+2025}. On the other hand, the achievability result comes from adapting the algorithm proposed in \cite{Gaudio+2025}. The algorithm has three main steps: (1) label an initial seed set of $\Theta(\log n)$ vertices using the maximum a posteriori (MAP) estimator; (2) propagate labels among ``blocks'' of vertices to achieve almost exact recovery; and (3) refine the almost exact labeling to turn it into an exact labeling. Step 3 is where the IT threshold is invoked. The main challenge in adapting the algorithm to the current, more general setting is in Steps (1) and (2). By allowing the observation distributions to be distance-dependent, there are pairs of vertices $(u,v)$ such that their pairwise observation, given by $X_{uv}$, has a similar distribution no matter what communities $u$ and $v$ belong to. In this way, some distances are more informative than others. While the IT threshold accounts for the heterogeneity of information, accounting for this hurdle in Steps (1) and (2), where the IT threshold is not invoked, is the main technical challenge in this paper.

\subsection{Notation and Organization}
We use $D_+$ to denote the Chernoff-Hellinger (CH) divergence, which was first introduced by Abbe and Sandon in \cite{AbbeSandon2015a}, then extended in \cite{Gaudio+2025} and \cite{GaudioJin2025}. We use the definition from \cite{GaudioJin2025}, which goes as follows. Let $p(y) = (p_1(\cdot;y), \ldots, p_k(\cdot; y))$ and $q(y) = (q_1(\cdot; y), \ldots, q_k(\cdot ;y))$ be vectors of probability distributions depending on a parameter $y \in \calY$, where $p_i(\cdot; y)$ and $q_i(\cdot; y)$ are probability mass functions (in the discrete case) or probability density functions (in the continuous case). In addition, let $\pi = (\pi_1, \ldots, \pi_k)$ be a vector of prior probabilities, and let $g(y)$ be a density on $\mathcal{Y}$. Then, we define the CH-divergence between $p$ and $q$ with respect to $\pi$ and $g$ as
\begin{equation}
\label{eq:CH_divergence_discrete}
\begin{aligned}
    &D_+(p \Vert q; \pi, g) =  1 - \inf_{t \in [0,1]} \sum_{i = 1}^k \pi_i \int_{y \in \mathcal{Y}} g(y)\sum_{x \in \mathcal{X}} p_i(x;y)^t q_i(x;y)^{1-t} \: dy
\end{aligned}
\end{equation}
when $p$ and $q$ are discrete, and 
\begin{equation}
\label{eq:CH_divergence_continuous}
\begin{aligned}
    &D_+(p \Vert q; \pi, g) =  1 - \inf_{t \in [0,1]} \sum_{i = 1}^k \pi_i \int_{y \in \mathcal{Y}} g(y) \int_{x \in \mathcal{X}} p_i(x;y)^t q_i(x;y)^{1-t} \: dx \: dy
\end{aligned}
\end{equation}
when $p$ and $q$ are continuous. 

The rest of the paper is organized as follows. Section \ref{sec:model} formally defines the GHCM model and presents our main results establishing an information-theoretic threshold for exact recovery under certain assumptions. Section \ref{sec:exact_recovery_algorithm} describes the algorithm for exact recovery and provides a proof sketch for our achievability result. Section \ref{sec:proof_sketch_impossibility} provides a proof sketch for our impossibility result. We discuss potential future research directions in Section \ref{sec:discussion}. Appendix \ref{appx:full_exact_recovery_algorithm} presents the full details of the exact recovery algorithm and Appendix \ref{appx:proof_of_exact_recovery} proves that the algorithm achieves exact recovery, except for the special one-dimensional case where there's a unique ``permissible labeling''. Appendix \ref{appx:1d_proof_of_exact_recovery} handles that special case, providing a modified algorithm to achieve exact recovery and the corresponding proof. Finally, Appendix \ref{appx:proof_of_impossibility} proves the impossibility of exact recovery below the information-theoretic threshold.

\section{Model and Main Results}
\label{sec:model}
We consider a generalization of the GHCM \cite{Gaudio+2025} by allowing the edge weight distribution between two vertices to depend on their distance, in addition to their community labels.

\begin{definition}
Let $n \in \N$ be a scaling parameter. Let $\lambda, r > 0$ be constants, $d \in \N$ be the dimension, and $k \in \N$ be the number of communities. Let $Z$  with $|Z| = k$ be a set which contains the label of each community. Let $\pi \in \R^k$ represent the prior probabilities of each community. For any $i, j \in Z$ and $y \in \R$, let $P_{ij}(y)$ be the edge weight distribution between two vertices of community $i$ and $j$ at a distance $y$. Let $p_{ij}(\cdot ; y)$ denote the probability mass function (PMF) or probability density function (PDF) of $P_{ij}(y)$. Then, we sample $G \sim \text{GHCM}(\lambda, n, r, \pi, P(y), d)$ as follows:
\begin{enumerate}
    \item The locations of the vertices are generated by a Poisson point process with intensity $\lambda$ over $\calS_{d, n} := [-n^{1/d}/2, n^{1/d}/2]^d$, the cube of volume $n$ centered at the origin. Let $V$ be the set of generated vertices.
    
    \item Each vertex $v \in V$ is independently assigned a community label, denoted $\sigma^*(v)$, with $\P(\sigma^*(v) = i) = \pi_i$ for each $i \in Z$.
    
    \item Conditioned on the locations and community labels, edge weights are generated independently between each pair of vertices. Let $\overline{P}(y) := P(y/(\log n)^{1/d})$ and $\overline{p}(\cdot; y)$ be the PMF/PDF of $\overline{P}(y)$. For each pair of vertices $\{u, v\}$ where $u \neq v$, the edge weight $X_{uv}$ is sampled from
    \begin{equation*}
        X_{uv} \sim \overline{P}_{\sigma^*(u)\sigma^*(v) }(\|u-v\|),
    \end{equation*}
    where $\|\cdot\|$ denotes the Euclidean toroidal metric. We denote the observed edge weight as $x_{uv}$. We assume that $r := \inf\{y > 0 : P_{ij}(y) = 0 \: \forall y > r\} < \infty$, and call $r (\log n)^{1/d}$ the \textit{visibility radius}. In addition, we say that $u$ is \textit{visible} to $v$ if $\|u - v\| \leq r (\log n)^{1/d}$, which we denote $u \visible v$. 
\end{enumerate}
\end{definition}
Our goal is to achieve exact recovery, which is a particular notion of recovering the community labels. As discussed in \cite{Gaudio+2025}, note that if there are symmetries in the prior probabilities and edge weight distributions, then it is only possible to recover the correct labeling up to a permutation. Consequently, \cite{Gaudio+2025} introduced the definition of a permissible relabeling to account for such symmetries, which we modify to account for distance-dependent edge weight distributions. 
\begin{definition}[Permissible Relabeling]
A permutation of communities $\omega: Z \to Z$ is called permissible if $\pi_i = \pi_{\omega(i)}$ and $P_{ij}(y) = P_{\omega(i)\omega(j)}(y)$ for all $i, j \in Z$ and almost all $y \in [0, r]$. We denote the set of permissible relabelings as $\Omega_{\pi, P}$.
\end{definition}
Now, we define exact recovery. Let $\sigma^*_n$ be the true labeling and $\tilde{\sigma}_n$ be an estimated labeling. We define the agreement of $\tilde{\sigma}_n$ and $\sigma^*_n$ as
\begin{equation*}
    A(\tilde{\sigma}_n, \sigma^*_n) = \frac{1}{|V|} \max_{\omega \in \Omega_{\pi, P}} \sum_{u \in V} \1\{\tilde{\sigma}_n(u) = \omega \circ \sigma^*_n(u)\}
\end{equation*}
which is the proportion of vertices that $\tilde{\sigma}_n$ labels correctly, up to a permissible relabeling. Then, we say that $\tilde{\sigma}_n$ achieves
\begin{itemize}
    \item \textit{exact recovery} if $\lim_{n \to \infty} \P(A(\tilde{\sigma}_n, \sigma^*_n) = 1) = 1$,
    \item \textit{almost-exact recovery} if $\lim_{n \to \infty} \P(A(\tilde{\sigma}_n, \sigma^*_n) \geq 1 - \epsilon) = 1$, for all $\epsilon > 0$, and
    \item \textit{partial recovery} if $\lim_{n \to \infty} \P(A(\tilde{\sigma}_n, \sigma^*_n) \geq \alpha) = 1$ for some constant $\alpha > 0$.
\end{itemize}

We now state our impossibility result for exact recovery. Let $\theta_i := (p_{i1}, p_{i2}, \ldots, p_{ik})$ be the vector of edge weight distributions from community $i$ to all other communities, and $g(y) := dy^{d-1}/r^d$ be a density on $\calY = [0, r]$. Let $\nu_d$ be the volume of the d-dimensional unit ball, and $D_+$ be the CH-divergence defined in \eqref{eq:CH_divergence_discrete} and \eqref{eq:CH_divergence_continuous}.

\begin{theorem}[Impossibility]
\label{thm:impossibility}
    Any estimator $\tilde{\sigma}: V \to Z$ fails to achieve exact recovery for $G \sim \text{GHCM}(\lambda, n, r, \pi, P(y), d)$ when either
    \begin{enumerate}
        \item $\lambda \nu_d r^d \min_{i \neq j} D_+(\theta_i\Vert \theta_j ; \pi ,g) < 1$, or
        \item $d = 1, \lambda r < 1,$ and $|\Omega_{\pi, P}| \geq 2$.
    \end{enumerate}
\end{theorem}

Next, we present our achievability result for exact recovery. We require the following assumptions, which are similar to those in \cite{Gaudio+2025}. These assumptions are used to achieve almost-exact recovery in Steps (1) and (2) outlined in the introduction.
\begin{assumption}[Identifiability of pairwise distributions]
\label{assump:regularity_condition}
    If $\{i, j\} \neq \{i', j'\}$ as sets, then either
    \begin{itemize}
        \item $P_{ij}(y) \equiv P_{i'j'}(y)$ for almost all $y \in [0, r]$, or
        \item $P_{ij}(y) \not \equiv P_{i'j'}(y)$ for almost all $y \in [0, r]$. 
    \end{itemize}
    We abbreviate the cases as $P_{ij} \equiv P_{i'j'}$ and $P_{ij} \not \equiv P_{i'j'}$, respectively.
\end{assumption}

\begin{assumption}[Bounded log-likelihood difference]
\label{assump:bounded_log_likelihood}
We assume that $P_{ij}$ have the same support $\calX \subset \R$ for all $i, j \in Z$, and that there exists $\eta > 0$ such that 
\begin{equation*}
    \log\left( \frac{p_{ij}(x; y)}{p_{ab}(x; y)} \right) < \eta
\end{equation*}
for any $i, j, a, b \in Z$, any $x \in \calX$,  
and almost all $y \in [0, r]$.
\end{assumption}

\begin{assumption}[Distinctness]
\label{assump:distinctness}
For any community $i \in Z$, we have $P_{ia}(y) \not \equiv P_{ib}(y)$ for any $a, b \in Z$ where $a \neq b$ and almost all $y \in [0, r]$.
\end{assumption}

For exact recovery when $d = 1$ and $|\Omega_{\pi, P}| = 1$, we need a stronger distinctness assumption.

\begin{assumption}[Strong Distinctness]
\label{assump:strong_distinctness}
For any $i, j, a, b \in Z$, we have $P_{ij}(y) \not \equiv P_{ab}(y)$ for almost all $y \in [0, r]$.
\end{assumption}

Under these assumptions, we establish the following achievability theorem.

\begin{theorem}[Achievability]
\label{thm:achievability}
    Under Assumptions \ref{assump:regularity_condition}, \ref{assump:bounded_log_likelihood}, and \ref{assump:distinctness}, there exists a polynomial-time algorithm that achieves exact recovery for $G \sim \text{GHCM}(\lambda, n, r, \pi, P(y), d)$ if $\lambda \nu_d r^d \min_{i \neq j} D_+(\theta_i \Vert \theta_j; \pi, g) > 1$ and either
    \begin{enumerate}
        \item $d \geq 2$, or 
        \item $d = 1, \lambda r > 1$.
    \end{enumerate}
    Under Assumptions \ref{assump:regularity_condition}, \ref{assump:bounded_log_likelihood} and \ref{assump:strong_distinctness}, there exists a polynomial-time algorithm that achieves exact recovery for $G \sim \text{GHCM}(\lambda, n, r, \pi, P(y), d)$ if $\lambda \nu_d r^d \min_{i \neq j} D_+(\theta_i \Vert \theta_j; \pi, g) > 1$, $d = 1$, and $|\Omega_{\pi, P}| = 1$.
\end{theorem}

\section{Exact Recovery Algorithm}
\label{sec:exact_recovery_algorithm}
\subsection{The Algorithm}
In this section, we describe our algorithm for exact recovery for the $d \geq 2$ and $d = 1$, $\lambda r > 1$ cases. We note that our exact recovery algorithm is the same as the exact recovery algorithm in \cite{Gaudio+2025}; our main contribution is proving that this algorithm achieves exact recovery even in the distance-dependent GHCM under suitable assumptions. To describe the algorithm, we need the following definitions from \cite{Gaudio+2025}.

\begin{definition}[Occupied Block]
    Given $\delta > 0$, we call a block $B_i$ $\delta$-occupied if $B_i$ contains at least $\delta \log n$ vertices.
\end{definition}

\begin{definition}[Mutually Visible Blocks]
    We call two blocks $B_i$ and $B_j$ mutually visible, which we note $B_i \visible B_j$, if
    \begin{equation*}
        \sup_{u \in B_i, v \in B_j} \|u -v\| \leq r(\log n)^{1/d}.
    \end{equation*}
\end{definition}

\begin{definition}[Block Visibility Graph]
    Let $G = (V, E)$ be a graph on $\calS_{d, n}$ and consider a partition of $\calS_{d, n}$ into blocks of volume $v(n)$. Then, we define the $(v(n), c(n))$-block visibility graph of $G$ as the graph $H = (V^\dagger, E^\dagger)$, where $V^\dagger := \{i \in [n/v(n)] : |V_i| \geq c(n)\}$ is the set of all blocks with at least $c(n)$ vertices and $E^\dagger := \{\{i, j\}: i, j \in V^\dagger, B_i \visible B_j\}$ is the set of all pairs of blocks in $V^\dagger$ that are mutually visible.
\end{definition}

We now provide an overview of the algorithm. First, we divide $\calS_{d, n}$ into cubes of volume $r^d \chi \log n$, which we call \textit{blocks}. The parameters $\chi$ and $\delta$ are chosen so that the $(r^d \chi \log n, \delta \log n)$-block visibility graph is connected with high probability.
If $d = 1$, fix $\chi_0 < (1 - 1/(\lambda r)) / 2$ and $\delta'$ such that Proposition 1 in \cite{GaudioJin2025} is satisfied. Then, we can take any \begin{equation}
\label{eq:chi_delta_conditions_d=1}
    \chi_0/2 < \chi < \chi_0  \: \text{ and } \: \delta' \chi / 2 < \delta < \delta' \chi,
\end{equation}
where $\delta'$, and thus $\delta$, can be made arbitrarily small. If $d = 2$, fix $\chi_0$ such that
\begin{equation*}
    \lambda r^d \bigg(\nu_d \Big(1-\frac{3\sqrt{d}}{2} \chi_0^{1/d} \Big)^d - \chi_0 \bigg) > 1 \: \text{ and } \: 1-\frac{3\sqrt{d}}{2} \chi_0^{1/d} > 0
\end{equation*}
and $\delta''$ such that Proposition 1 in \cite{GaudioJin2025} is satisfied. Then, we can take any 
\begin{equation}
\label{eq:chi_delta_conditions_d>=2}
    \chi_0/2 < \chi < \chi_0 \: \text{ and } \: \delta''\chi / (2 \nu_d) < \delta < \delta''\chi / \nu_d,
\end{equation}
where $\delta''$, and thus $\delta$, can be made arbitrarily small. Since the block visibility graph is connected, we can then label the vertices block-by-block, for all $\delta$-occupied blocks. In particular, we find a rooted spanning tree of the block visibility graph, ordering $V^\dagger = \{i_1, i_2, \ldots, i_{|V^\dagger|}\}$ in breadth-first order. For any $i_j \in V^\dagger$, we denote its parent block as $p(i_j)$.

First, we label $V_{i_0}$ (a subset of vertices in the initial block $V_{i_1}$) using the MAP estimator, which computes the labeling that maximizes the posterior likelihood of the observed graph $G$ (i.e. the vertex locations and pairwise edge observations). That is,
\begin{equation}
\label{eq:MAP}
    \hat{\sigma}_{V_{i_0}} = \argmax_{\sigma:V_{i_0} \to Z} \P(\sigma^* = \sigma \mid G).
\end{equation}
We remark that we only label $V_{i_0}$ instead of the entire initial block to maintain runtime efficiency, since the MAP estimator is computed through an exhaustive iteration of all possible labelings. 

Then, we iterate through each $\delta$-occupied block $V_{i_j}$ for $j = 1, \ldots, |V^\dagger|$, labeling each block using the estimated labels of the parent block. We call this technique \textit{propagation}, which was first developed in \cite{Gaudio+2024}. For a given block $V_{i_j}$, we label each vertex $v \in V_{i_j}$ with the community that maximizes the posterior likelihood of the observed edge weights between $v$ and the vertices $V_{p(i_j)}$ in the parent block. That is,
\begin{equation}
\label{eq:propagate}
    \hat{\sigma}(v) = \argmax_{a \in Z} \sum_{u \in V_{p(i_j)}: \hat{\sigma}(u) = b} \log(p_{ab}(x_{uv}; \|u - v\|)).
\end{equation}
Due to the connectivity of the block visibility graph, we can label all $\delta$-occupied blocks via propagation. This produces an almost-exact labeling $\hat{\sigma}: V \to Z$ 
 
Finally, we compute the exact labeling $\tilde{\sigma}: V \to Z$ iteratively for each vertex $v \in V$. Given the initial labeling $\hat{\sigma}$, we compute $\tilde{\sigma}(v)$ by choosing the community which maximizes the posterior likelihood of the observed graph $G$ with respect to the labeling $\hat{\sigma}$. That is,
\begin{equation}
\label{eq:refine_label}
    \tilde{\sigma}(v) = \argmax_{j \in Z} \sum_{u \in V: u \visible v} \log( p_{j, \hat{\sigma}(u)}(x_{uv}; \|u - v\|) ).
\end{equation}
This procedure is inspired by the genie-aided estimator, which labels a vertex knowing the true labels of all other vertices.

We leave the full description of the algorithm for the $d = 1, |\Omega_{\pi, P}| = 1$ case to Appendix \ref{appx:1d_proof_of_exact_recovery}. In summary, the algorithm is quite similar---we first divide $S_{1, n}$ into blocks (though a different number of them), then label the initial blocks using the MAP estimator, label the remaining blocks via propagation, and compute the exact labeling as in \eqref{eq:refine_label}. The main difference is that we have multiple initial blocks because the block visibility graph may be disconnected. \footnote{The reason $d=1$ is a special case is tied to the fact that in a Random Geometric Graph, the connectivity threshold and isolated vertex thresholds coincide for $d \geq 2$ but differ for $d =1$ \cite{Penrose2003} \cite{Zhao2015}.}

\subsection{Proof Sketch}
First, we discuss the proof of exact recovery in the $d \geq 2$ and $d = 1, \lambda r > 1$ cases. Our proof mainly follows the strategy used to prove exact recovery in \cite{Gaudio+2025}, with several adaptions to handle distance-dependent edge weight distributions. First, we note that due to Proposition 1 in \cite{GaudioJin2025}, the choice of $\chi$ and $\delta$ in \eqref{eq:chi_delta_conditions_d=1} and \eqref{eq:chi_delta_conditions_d>=2} ensure that the $(r^d \chi \log n, \delta \log n)$-block visibility graph is connected with high probability. 

Next, we show that the MAP estimator correctly labels all vertices in $V_{i_0}$ with high probability, following the strategy of \cite{Gaudio+2025}. Rather than analyzing the MAP estimator directly, we analyze the restricted maximum likelihood estimator (MLE) because the analysis is more tractable. To define the restricted MLE, let 
\begin{equation}
\begin{aligned}
\label{eq:restricted_labeling_set}
    X_0^*(\epsilon) := &\Big\{\sigma : V_{i_0} \to Z \text{ such that } |\{u \in V_{i_0} : \sigma(u) = j\}| \in \Big((\pi_j - \epsilon)|V_{i_0}|, (\pi_j + \epsilon)|V_{i_0}|\Big) \: \forall j \in Z\Big\},
\end{aligned}
\end{equation}
which is the set of all labelings such that the proportion of vertices in each community is close to the expected proportion. Then, the restricted MLE $\bar{\sigma}$ chooses the labeling within $X_0^*(\epsilon)$ that maximizes the likelihood of the observed graph. That is,
\begin{equation}
\label{eq:restricted_MLE}
    \bar{\sigma} = \argmax_{\sigma \in X_0^*(\epsilon)} \sum_{u \in V_{i_0}} \sum_{v \in V_{i_0}, v \neq u} \log\Big(\overline{p}_{\sigma(u), \sigma(v)}(x_{uv}; \|u -v\|)\Big).
\end{equation}
We then show that for sufficient small $\epsilon$, the restricted MLE labels all vertices in $V_{i_0}$ correctly with high probability. Due to optimality of the MAP estimator, the MAP labels all vertices in $V_{i_0}$ correctly with high probability.

To accomplish this, we fix $\epsilon > 0$ sufficiently small and upper-bound the probability that there exists an incorrect labeling (up to a permissible relabeling) in $X_0^*(\epsilon)$ having a higher posterior likelihood than the ground truth labeling $\sigma^*$---this is the error probability of the MAP estimator. We adapt an idea of Dhara, Gaudio, Mossel, and Sandon \cite{Dhara+2024}, dividing the proof into two cases depending on the \textit{discrepancy} between the incorrect labeling $\sigma$ and $\sigma^*$, which is the number of vertices on which $\sigma$ differs from $\sigma^*$. In the low-discrepancy case, the proof relies on concentration inequalities to directly obtain a high probability upper bound over all low-discrepancy labelings. Meanwhile, in the high-discrepancy case, we first use concentration inequalities to upper-bound the probability for any given high-discrepancy labeling and then use the union bound. We note that Assumptions \ref{assump:regularity_condition}, \ref{assump:bounded_log_likelihood}, and \ref{assump:distinctness} are used here.

Then, we show that during propagation, $\hat{\sigma}$ makes at most $M$ mistakes in each $\delta$-occupied block $V_{i_j}$ for $j = 1, \ldots, |V^\dagger|$, for some constant $M$. We follow the strategy used in \cite{Gaudio+2025} and first use concentration inequalities to show that $\hat{\sigma}$ makes more than $M$ mistakes on any given block with probability $o(1/n)$. Then, we show that $\hat{\sigma}$ makes at most $M$ mistakes on all blocks with high probability using a straightforward calculation, which establishes that $\hat{\sigma}$ achieves almost-exact recovery. We remark that Assumptions \ref{assump:bounded_log_likelihood} and \ref{assump:distinctness} are required here.

Finally, we show that labeling $\tilde{\sigma}$ achieves exact recovery using the approach of \cite{Gaudio+2025}. In particular, we show that the MAP estimator with respect to the initial labeling $\hat{\sigma}$ makes a mistake on a given vertex $v \in V$ with probability $o(1/n)$. To accomplish this, we first show that the MAP estimator with respect to the true labels comes ``close'' to making a mistake on a given vertex $v \in V$ with probability $o(1/n)$. Formally, let $\ell_i(v, \sigma)$ denote the posterior log-likelihood of $v$ having community label $i$ when the remaining vertices are labeled using $\sigma$. That is
\begin{equation}
\label{def:ell}
    \ell_i(v, \sigma) = \sum_{u \in V: u \visible v} \log(\overline{p}_{i, \sigma(u)}(x_{uv};\|u - v\|)).
\end{equation}
Now, suppose that $\sigma^*(v) = i$ and $j \neq i$. We show that when $\lambda \nu_d r^d \min_{i \neq j}D_+(\theta_i \Vert \theta_j; \pi, g) > 1$ (i.e. we are above the information-theoretic threshold), there exists $\epsilon > 0$ such that
\begin{equation}
\label{eq:refine_eq1}
    \P(\ell_i(v, \sigma^*) - \ell_j(v, \sigma^*) \leq \epsilon \log n \mid \sigma^*(v) = i) = o(1/n). 
\end{equation}
Then, we show that the MAP estimator under the initial labeling $\hat{\sigma}$ is close to the MAP estimator under the true labels for all vertices with high probability, meaning that
\begin{equation}
\label{eq:refine_eq2}
    |\ell_i(v, \sigma^*) - \ell_i(v, \hat{\sigma})| \leq \epsilon (\log n) / 2 \:
\end{equation}
for all $v \in V$ with high probability. Combining \eqref{eq:refine_eq1} and \eqref{eq:refine_eq2} shows that $\hat{\sigma}$ makes a mistake on $v$ with probability $o(1/n)$. The union bound then gives us that $\hat{\sigma}$ makes a mistake with probability $o(1)$, therefore achieving exact recovery. We note that Assumption \ref{assump:bounded_log_likelihood} is used here.

We leave the proof for $d = 1, |\Omega_{\pi, P}| = 1$ case to Appendix \ref{appx:1d_proof_of_exact_recovery}; the structure of the proof is the same, but the details are somewhat different.

\section{Proof Sketch of Impossibility}
\label{sec:proof_sketch_impossibility}
In this section, we discuss the impossibility proof, starting with the case where $\lambda \nu_d r^d \min_{i \neq j} D_+(\theta_i \Vert \theta_j ; \pi ,g) < 1$. We split the proof into two further cases: $\lambda \nu_d r^d < 1$ and $\lambda \nu_d r^d \geq 1$. For the first case where $\lambda \nu_d r^d < 1$, we use Theorem 7.1 from Penrose \cite{Penrose2003} to show that there is an isolated vertex with high probability. Consequently, there is a constant nonzero probability that the MAP estimator fails, since it fails to label the isolated vertex correctly with a constant nonzero probability. Since the MAP is optimal, any other estimator fails as well. For the second case where $\lambda \nu_d r^d \geq 1$, we use Proposition 8.1 from \cite{Abbe+2020}.

\begin{lemma}[Proposition 8.1 in \cite{Abbe+2020}]
\label{lem:impossibility_lemma}
    We call a vertex $v \in V$ ``Flip-Bad'' if there exists a community $j \neq \sigma^*(v)$ such that $\ell_j(v, \sigma^*) \geq \ell_{\sigma^*(v)}(v, \sigma^*)$. Let $F_v^H$ denote the event that vertex $v$ is Flip-Bad in graph $H$. Now, suppose that the graph $G$ satisfies the following two conditions:
    \begin{gather}
    \label{eq:impossibility_condition_1} \lim_{n\to\infty} n\E^{0}[\1_{F_0^{G \cup \{0\}}}] = \infty, \text{ and } \\
    \label{eq:impossibility_condition_2} \limsup_{n\to\infty} \frac{\int_{y \in \calS_{d,n}} \E^{0,y}[\1_{F_0^{G \cup \{0, y\}}} \1_{F_y^{G \cup \{0, y\}}} ] \: m_{n,d}(dy)}{n\E^{0}[\1_{F_0^{G \cup \{0\}}}]^2} \leq 1
    \end{gather}
    where $m_{n,d}$ is the Haar measure, $\E^0$ represents the expectation with respect to the graph $G \cup \{0\}$, and $\E^{0,y}$ is the expectation with respect to the graph $G \cup \{0, y\}$. Then, exact recovery is impossible.
\end{lemma}

To prove \eqref{eq:impossibility_condition_1}, we adapt the proof of Lemma B.4 in \cite{Gaudio+2025}. We lower-bound the probability that the vertex $0$ is Flip-Bad; in particular, we show that $\P(0 \text{ is Flip-Bad in } G \cup \{0\}) \geq  n^{-1 + \beta}$ for some constant $\beta > 0$, which implies \eqref{eq:impossibility_condition_1}. To compute this lower bound, we use Cram\'er's Theorem from large deviations theory. We note that the condition $\lambda \nu_d r^d \min_{i \neq j} D_+(\theta_i\Vert \theta_j ; \pi ,g) < 1$ (i.e. being below the information-theoretic threshold) is used here. To prove $\eqref{eq:impossibility_condition_2}$, we use a straightforward calculation that is essentially identical to that of Lemma B.5 in \cite{Gaudio+2025}.

We now discuss the impossibility proof when $d = 1, \lambda r < 1, \text{ and } |\Omega_{\pi, P}| \geq 2$. We define the \textit{vertex visibility graph} of $G$ as the graph $G' = (V, E')$ where $E' = \{\{u, v\}: u,v \in V, \|u - v\| \leq r(\log n)^{1/d} \}$, i.e. the graph formed from taking all vertices of $G$ and adding an edge between any pair of vertices which are visible to each other. From Lemma 15 in \cite{GaudioJin2025}, we obtain that the vertex visibility graph of $G$ is disconnected with high probability. This makes exact recovery impossible because we can permute the labels in one connected component (but not other components) via a non-trivial permissible relabeling, which yields an incorrect labeling with the same posterior likelihood as $\sigma^*$. Since the MAP estimator cannot differentiate between this incorrect labeling and $\sigma^*$, the MAP will be incorrect with constant probability.

\section{Discussion}
\label{sec:discussion}
In this paper, we identified the information-theoretic threshold for exact recovery in the distance-dependent GHCM under some technical conditions. In addition, we provided an efficient algorithm which achieves exact recovery above this threshold.

One limitation of this paper is that we do not handle the case where two pairwise distributions $P_{ij}$, $P_{ab}$ are equal for some distances and not equal for other distances. This case would be challenging for the MAP estimator and for propagation if, for example, the distributions coincide nearby.
Another related direction for future work is considering the case where the pairwise distributions have different ``radii''. For instance, labels $(i, j)$ could have information up to a larger radius $r_1$ while $(i', j')$ only have information up to a smaller radius $r_2$.

In addition, we assume that the vertex locations and parameters of the model are known. It would be interesting to consider the question where the distribution family for the pairwise distributions is known, but the parameters are not; this question has been studied in the SBM, for instance by Abbe and Sandon \cite{AbbeSandon2015b}, but not for spatially-embedded community models such as the GSBM or GHCM.

\bibliographystyle{plain}
\bibliography{Refs}

@article{Gaudio+2025,
author = {Julia Gaudio and Charlie Guan and Xiaochun Niu and Ermin Wei},
title = {{Exact Label Recovery in Euclidean Random Graphs}},
journal = {arXiv:2407.11163 [cs.SI]},
year = {2025}
}

@article{AbbeSandon2015a,
author = {Emmanuel Abbe and Colin Sandon},
title = {{Community detection in general stochastic block models: fundamental limits and efficient algorithms for recovery}},
journal = {2015 IEEE 56th Annual Symposium on Foundations of Computer Science},
pages = {670--688},
year = {2015}
}

@article{GaudioJin2025,
author = {Julia Gaudio and Andrew Jin},
title = {{Exact Recovery in the Geometric SBM}},
journal = {arXiv:2512.22773},
year = {2025}
}

@book{Penrose2003,
author = {Mathew Penrose},
title = {{Random Geometric Graphs}},
year = {2003},
publisher = {Oxford University Press}
}

@article{Abbe+2020,
author = {Emmanuel Abbe and Francois Baccelli and Abishek Sankararaman},
title = {{Community detection on Euclidean random graphs}},
journal = {Information and Inference: A Journal of the IMA},
volume = {10},
number = {1},
pages = {109--160},
year = {2021}
}

@article{Dhara+2024,
author = {Souvik Dhara and Julia Gaudio and Elchanan Mossel and Colin Sandon},
title = {The Power of Two Matrices in Spectral Algorithms for Community Recovery},
journal = {IEEE Transactions on Information Theory},
volume = {70},
number = {5},
pages = {3599--3621},
year = {2024}
}

@article{Gaudio+2024,
author = {Julia Gaudio and Xiaochun Niu and Ermin Wei},
title = {{Exact Community Recovery in the Geometric SBM}},
journal = {Proceedings of the 2024 Annual ACM-SIAM Symposium on Discrete Algorithms (SODA)},
pages = {2158--2184},
year = {2024}
}

@article{Holland1983,
  title={Stochastic blockmodels: First steps},
  author={Holland, Paul W and Laskey, Kathryn Blackmond and Leinhardt, Samuel},
  journal={Social networks},
  volume={5},
  number={2},
  pages={109--137},
  year={1983},
  publisher={Elsevier}
}

@inproceedings{Sankararaman2018,
  title={Community detection on {Euclidean} random graphs},
  author={Sankararaman, Abishek and Baccelli, Fran{\c{c}}ois},
  booktitle={Proceedings of the Twenty-Ninth Annual ACM-SIAM Symposium on Discrete Algorithms},
  pages={2181--2200},
  year={2018},
  organization={SIAM}
}

@inproceedings{GaudioGuan2025,
  title={Sharp exact recovery threshold for two-community {Euclidean} random graphs},
  author={Gaudio, Julia and Guan, Charlie K},
  booktitle={International Symposium on Information Theory (ISIT)},
  year={2025},
  organization={IEEE}
}

@article{Avrachenkov2024,
  title={Community detection on block models with geometric kernels},
  author={Avrachenkov, Konstantin and Kumar, BR and Leskel{\"a}, Lasse},
  journal={arXiv preprint arXiv:2403.02802},
  year={2024}
}

@article{Zhao2015,
title = {{The Absence of Isolated Node in Geometric Random Graphs}},
author = {Jun Zhao},
journal = {2015 53rd Annual Allerton Conference on Communication, Control, and Computing},
year = {2015}
}

@article{Abbe2018,
  title={Community Detection and Stochastic Block Models: Recent Developments},
  author={Abbe, Emmanuel},
  journal={Journal of Machine Learning Research},
  volume={18},
  number={177},
  pages={1--86},
  year={2018}
}

@inproceedings{AbbeSandon2015b,
title = {Recovering Communities in the General Stochastic
Block Model Without Knowing the Parameters
},
author = {Emmanuel Abbe and Colin Sandon},
booktitle = {Advances in Neural Information Processing Systems},
year= {2015}
}

@book{Lugosi+2013,
title = {{Concentration Inequalities: A Nonasymptotic Theory of Independence}},
author = {St{\'e}phane Boucheron and G{\'a}bor Lugosi and Pascal Massart},
year = {2013},
publisher = {Oxford University Press}
}

\appendix
\section{Preliminaries}
We record some useful definitions, concentration inequalities, and results from large deviations here.

\begin{lemma}[Chernoff Bound, Poisson \cite{Lugosi+2013}]
\label{lem:poisson_Chernoff}
    Let $X \sim \text{Pois}(\mu)$ where $\mu > 0$. Then, for any $t > 0$,
    \begin{equation*}
        \P(X \geq \mu + t) \leq \exp\left(-\frac{t^2}{2(\mu + t)}\right)
    \end{equation*}
    and
    \begin{equation*}
        \P(X \leq \mu - t) \leq \exp \left(- (\mu - t) \log \left(1 - \frac{t}{\mu}\right) -t \right).
    \end{equation*}
\end{lemma}

\begin{lemma}[Chernoff Bound, Binomial]
\label{lem:binomial_Chernoff}
    Let $X_i: i=1, \ldots, n$ be independent Bernoulli random variables. Let $X = \sum_{i=1}^n X_i$ and $\mu = \E[X]$. Then, for any $t > 0$,
    \begin{equation*}
        \P(X \geq (1 + t)\mu) \leq \left(\frac{e^t}{(1+t)^{(1+t)}}\right)^\mu.
    \end{equation*}
    Also, for any $0 < t < 1$,
    \begin{equation*}
        \P(X \leq (1 - t)\mu) \leq \left(\frac{e^{-t}}{(1-t)^{(1-t)}}\right)^\mu
    \end{equation*}
\end{lemma}

\begin{definition}[Rate Function]
    Let $X$ be any random variable. We define the rate function of $X$, denoted $\Lambda_X^*$, as
    \begin{equation*}
        \Lambda_X^*(\alpha) := \sup_{t \in \R} (t\alpha - \Lambda_X(t))
    \end{equation*}
    where $\Lambda_X(t) := \log \E[\exp(tX)]$ is the cumulant generating function of $X$.
\end{definition}

\begin{lemma}[Cram\'er's Theorem]
\label{lem:Cramer}
Let $\{X_i\}$ be a sequence of i.i.d. random variables with rate function $\Lambda_X^*$. Then, for all $\alpha > \E[X_1]$,
    \begin{equation*}
        \lim_{m \to \infty} \frac{1}{m} \log \P(\sum_{i=1}^m X_i \geq m\alpha) = -\Lambda^*_X(\alpha).
    \end{equation*}
\end{lemma}

Suppose that $P(y)$ and $Q(y)$ are two sets of distributions indexed by a parameter $y \in [0, r]$, with corresponding PMFs/PDFs $p(\cdot; y)$ and $q(\cdot; y)$. We define
\begin{equation*}
    \phi_t(P(y), Q(y)) := \E_{X \sim q} \bigg[\bigg(\frac{p(X;y)}{q(X;y)}\bigg)^t\bigg]
\end{equation*}
and 
\begin{equation*}
    \overline{\phi}_t(P, Q) := \int_0^r \phi_t(P(y),Q(y) ) \frac{dy^{d-1}}{r^d} dy.
\end{equation*}
Furthermore, for any two communities $i, j \in Z$ we define 
\begin{equation}
\label{def:t_ij}
    t_{ij} := \argmin_{t \in [0, 1]} \sum_{a \in Z} \pi_a \overline{\phi}_t(P_{ia}, P_{ja})
\end{equation}
Using these definitions, we can express the CH-divergence between $\theta_i$ and $\theta_j$ (defined just before the statement of Theorem \ref{thm:impossibility}) as 
\begin{equation}
\label{eq;CH_divergence_alternate}
D_+(\theta_i \Vert \theta_j; \pi, g) = 1 - \sum_{a \in Z} \pi_a \overline{\phi}_{t_{ij}}(P_{ia}, P_{ja}).
\end{equation}

\FloatBarrier
\section{Exact Recovery Algorithm}
\label{appx:full_exact_recovery_algorithm}
For completeness, we record the exact recovery algorithm, which is a minor adaptation of \cite{Gaudio+2025}. Algorithm \ref{alg:exact_recovery} relies on three subroutines given by Algorithms \ref{alg:MAP}, \ref{alg:propagate}, and \ref{alg:Refine}.

\begin{algorithm}
\caption{Exact Recovery for $d \geq 2$ and $d = 1$, $\lambda r > 1$}
\label{alg:exact_recovery}
\begin{algorithmic}[1]
\Require $G \sim \text{GHCM}(\lambda, n, r, \pi, P(y), d)$.
\Ensure An estimated community labeling $\tilde{\sigma}: V \to Z$.
\State \textbf{Phase I:}
\State Take $\chi, \delta > 0$ which satisfy \eqref{eq:chi_delta_conditions_d=1} if $d = 1$ and \eqref{eq:chi_delta_conditions_d>=2} if $d \geq 2$. \label{line:pick_chi_delta}
\State Partition $\calS_{d, n}$ into $n / (r^d \chi \log n)$ blocks of volume $r^d \chi \log n$. Let $B_i$ be the $i$th block and $V_i$ be the set of vertices in $B_i$ for $i \in [n / (r^d \chi \log n)]$. \label{line:construct_blocks}
\State Construct the $(r^d \chi \log n, \delta \log n)$-visibility graph $H = (V^\dagger, E^\dagger)$ of $G$. \label{line:construct_visibility_graph}
\If{$H$ is disconnected}
    \State Return \texttt{FAIL}.
\EndIf
\State Find a rooted spanning tree of $H$, ordering $V^\dagger = \{i_1, i_2, \ldots, i_{|V^\dagger|}\}$ in breadth-first order.
\State Set $\epsilon_0 \leq \min\{1/(2 \log k), \delta\}$. Sample $V_{i_0} \subset V_{i_1}$ such that $|V_{i_0}| = \epsilon_0 \log n$. Set $V_{i_1} \gets V_{i_1} \setminus V_{i_0}$.
\State Apply \texttt{Maximum a Posteriori} (Algorithm \ref{alg:MAP}) on input $(G, V_{i_0})$ to obtain the labeling $\hat{\sigma}$ on $V_{i_0}$.
\For{$j = 1, 2, \ldots, |V^\dagger|$}
    \State Apply \texttt{Propagate} (Algorithm \ref{alg:propagate}) on input $(G, V_{p(i_j)}, V_{i_j}, \hat{\sigma})$ to obtain the labeling $\hat{\sigma}$ on $V_{i_j}$.
\EndFor
\For{$v \in V \setminus \cup_{i \in V^\dagger} V_i$}
    \State Set $\hat{\sigma}(v) = *$.
\EndFor
\State \textbf{Phase II:}
\For{$v \in V$}
    \State Apply \texttt{Refine} (Algorithm \ref{alg:Refine}) on input $(G, v, \hat{\sigma})$ to compute $\tilde{\sigma}(v)$.
\EndFor
\end{algorithmic}
\end{algorithm}

\begin{algorithm}
\caption{\texttt{Maximum a Posteriori (MAP)}}
\label{alg:MAP}
\begin{algorithmic}[1]
\Require Graph $G = (V, E)$ and vertex set $S \subset V$.
\Ensure An estimated labeling $\hat{\sigma}_S: S \to Z$.
\State Set 
\begin{equation*}
    \hat{\sigma}_S = \argmax_{\sigma:S \to Z} \P(\sigma^* = \sigma \mid G).
\end{equation*}
\end{algorithmic}
\end{algorithm}

\begin{algorithm}
\caption{\texttt{Propagate}}
\label{alg:propagate}
\begin{algorithmic}[1]
\Require Graph $G = (V, E)$, disjoint vertex sets $S, S' \subset V$ which are mutually visible, and labeling $\hat{\sigma}_S : S \to Z$.
\Ensure An estimated labeling $\hat{\sigma}_{S'}: S' \to Z$.
\State Let $i := \argmax_{j \in Z} |\{u \in S: \hat{\sigma}_S(u) = j\}|$ be the largest community in $S$.
\For{$v \in S'$}
    \State Set 
\begin{equation*}
    \hat{\sigma}_{S'}(v) = \argmax_{j \in Z} \sum_{u \in S: \hat{\sigma}(u) = i} \log(p_{ij}(x_{uv}; \|u - v\|)).
\end{equation*}
\EndFor
\end{algorithmic}
\end{algorithm}

\begin{algorithm}
\caption{\texttt{Refine}}
\label{alg:Refine}
\begin{algorithmic}[1]
\Require $G \sim \text{GHCM}(\lambda, n, r, \pi, P(y), d)$, vertex $v \in V$, labeling $\hat{\sigma}:V \to Z \cup \{*\}$.
\Ensure An estimated labeling $\tilde{\sigma}(v) \in Z$.
\State Set 
\begin{equation*}
    \tilde{\sigma}(v) = \argmax_{j \in Z} \sum_{u \in V: u \visible v} \log( p_{j, \hat{\sigma}(u)}(x_{uv}; \|u - v\|) )
\end{equation*}
\end{algorithmic}
\end{algorithm}
\FloatBarrier

\section{Proof of Exact Recovery}
\label{appx:proof_of_exact_recovery}
In this section, we will show that the Algorithm \ref{alg:exact_recovery} achieves exact recovery. We first provide a preliminary result that we will use throughout this section. Recall that we use the notation $P_{ij} \not \equiv P_{ab}$ to mean $P_{ij}(y) \not \equiv P_{ab}(y)$ for almost all $y \in [0, r]$.
\begin{lemma}
\label{lem:phi_property}  
    For any fixed $i, j, a, b \in Z$ such that $P_{ij} \not \equiv P_{ab}$, we define the set
    \begin{equation}
    \label{eq:def_S_prime}
        S'(\gamma; P_{ij}, P_{ab}, t) :=  \{y \in [0, r] : \phi_t(P_{ij}(y), P_{ab}(y) ) \geq 1 - \gamma\},
    \end{equation}
    which is the set of $y \in [0, r]$ such that $\phi_t(P_{ij}(y), P_{ab}(y))$ is close to 1. Then, we define
    \begin{equation}
    \label{eq:def_S}
        S(\gamma; t) := \bigcup_{\substack{i,j,a,b \in Z \\ P_{ij} \not \equiv P_{ab}}} S'(\gamma; P_{ij}, P_{ab}, t) = \bigcup_{\substack{i,j,a,b \in Z \\ P_{ij} \not \equiv P_{ab}}} \{y \in [0, r] : \phi_t(P_{ij}(y), P_{ab}(y) ) \geq 1 - \gamma\}.
    \end{equation}
    Then, for any $\epsilon > 0$ and $t \in (0, 1)$, there exists $\gamma > 0$ such that $|S(\gamma;t)| < \epsilon$.
\end{lemma}
\begin{proof}
    First, we show that for any $P_{ij} \not \equiv P_{ab}$, we have $\lim_{\gamma \to 0} |S'(\gamma; P_{ij}, P_{ab}, t)| = 0$. Let $\gamma_n$ be a monotone decreasing sequence such that $\lim_{n \to \infty} \gamma_n = 0$. Observing that $S'(\gamma_n; P_{ij}, P_{ab}, t)$ is a monotone decreasing sequence with $S'(\gamma_1; P_{ij}, P_{ab}, t) < \infty$, we have
    \begin{equation*}
        \bigg|\bigcap_{n=1}^\infty S'(\gamma_n; P_{ij}, P_{ab}, t)\bigg| = \lim_{n \to \infty} |S'(\gamma_n; P_{ij}, P_{ab}, t)|
    \end{equation*}
    by continuity of measure. Computing the left-hand side, we obtain that
    \begin{equation*}
    \begin{aligned}
        \bigcap_{n=1}^\infty S'(\gamma_n; P_{ij}, P_{ab}, t) &= \bigcap_{n=1}^\infty \{y \in [0, r]: \phi_t(P_{ij}(y), P_{ab}(y)) \geq 1 - \gamma_n\} \\
        &= \{y \in [0, r]: \phi_t(P_{ij}(y), P_{ab}(y)) \geq 1 - \gamma_n \: \forall n \in \N\} \\
        &= \{y \in [0, r]: \phi_t(P_{ij}(y), P_{ab}(y)) = 1\} \\
    \end{aligned}
    \end{equation*}
    which has measure zero since $P_{ij}(y) \not\equiv P_{ab}(y)$ for almost all $y \in [0, r]$. Hence, we obtain that the limit $\lim_{n \to \infty} |S'(\gamma_n; P_{ij}, P_{ab}, t)| = 0$. Since this is true for any sequence $\gamma_n \to 0$, we therefore have that $\lim_{\gamma \to 0} |S'(\gamma; P_{ij}, P_{ab}, t)| = 0$. Finally, observe that
    \begin{align*}
        \lim_{\gamma \to 0} |S(\gamma; t)| \leq \lim_{\gamma \to 0} \sum_{\substack{i,j,a,b \in Z \\ P_{ij} \not \equiv P_{ab}}} |S'(\gamma; P_{ij}, P_{ab}, t)| = 0,
    \end{align*}
    which completes the proof.
\end{proof}

Lemma \ref{lem:phi_property} implies the following corollary. We let $B_{r,d}$ denote the $d$-dimensional ball of radius $r$ centered at the origin.
\begin{corollary}
\label{cor:distinguishing_vertices}
Let $A$ be a fixed region contained within $B_{r,d}$. Let $D$ be the distance to the origin of a point placed uniformly at random in $A$. Then for any $\epsilon > 0$ and $t \in (0,1)$, there exists $\gamma > 0$ such that the probability that $\phi_t(P_{ij}(D), P_{ab}(D) ) \geq 1 - \gamma$ for some $i,j,a,b \in Z$ with $P_{ij} \not \equiv P_{ab}$ is at most $\epsilon$. Moreover, the value of $\gamma$ depends only on $\epsilon$, $t$, and the volume of $A$.
\end{corollary}
\begin{proof}
Let $E_A(x)$ be the event that $\phi_t(P_{ij}(D), P_{ab}(D) ) \geq 1 - x$ for some $i,j,a,b \in Z$ with $P_{ij} \not \equiv P_{ab}$.
We first consider the case where $A = B_{r,d}$. Fix $\epsilon > 0$ and $t \in (0,1)$, and let $\epsilon'$ satisfy $\epsilon = 1 - ((r-\epsilon')/r)^d$. Then by Lemma \ref{lem:phi_property}, there exists $\gamma > 0$ such that $|S(\gamma; t)| < \epsilon'$. In the worst case, $S(\gamma; t) = (r-\epsilon', r]$. Then $E_A(\gamma)$ holds with probability at most $1- ((r-\epsilon')/r)^d = \epsilon$, as desired.

If $A \neq B_{r,d}$, then again fix $\epsilon > 0$ and $t \in (0,1)$. Let $\epsilon' > 0$ satisfy $\epsilon = \epsilon' \cdot \nu_d r^d / |A|$. Then there exists $\gamma > 0$ such that $\mathbb{P}(E_{B_{r,d}}(\gamma)) \leq \epsilon'$. Since $|A| / |B_{r, d}| = |A| / (\nu_d r^d)$, it follows that 
\begin{equation*}
\mathbb{P}(E_{A}(\gamma)) \cdot \frac{|A|}{|B_{r,d}|} \leq \mathbb{P}(E_{B_{r,d}}(\gamma)) \leq \epsilon' 
\end{equation*}
and so
$\mathbb{P}(E_{A}(\gamma)) \leq \epsilon$. Since $\epsilon$, and therefore $\gamma$, depends on $A$ only through its volume, the claim follows.
\end{proof}

\subsection{Maximum a Posteriori}
\label{subappx:MAP_proof}
In this section, we show that the MAP estimator achieves exact recovery on $V_{i_0}$ with high probability, using ideas from \cite{Gaudio+2025}.  

\begin{prop}
\label{prop:MAP}
    Let $G \sim \text{GHCM}(\lambda, n, r, \pi, P(y), d)$ and let $\hat{\sigma}$ be the estimated labels for $V_{i_0}$. Suppose that Assumptions \ref{assump:regularity_condition}, \ref{assump:bounded_log_likelihood}, and \ref{assump:distinctness} hold. Then, $\hat{\sigma}$ achieves exact recovery on $V_{i_0}$ with high probability.
\end{prop}

Like in \cite{Gaudio+2025}, we analyze the restricted MLE rather than the MAP estimator directly. From \eqref{eq:restricted_labeling_set} and \eqref{eq:restricted_MLE}, recall that the definition of the restricted MLE is
\begin{equation*}
    \bar{\sigma} = \argmax_{\sigma \in X_0^*(\epsilon)} \sum_{u \in V_{i_0}} \sum_{v \in V_{i_0}, v \neq u} \log\Big(\overline{p}_{\sigma(u), \sigma(v)}(x_{uv}; \|u -v\|)\Big)
\end{equation*}
where
\begin{equation*}
    X_0^*(\epsilon) := \Big\{\sigma : V_{i_0} \to Z \text{ such that } |\{u \in V_{i_0} : \sigma(u) = j\}| \in \Big((\pi_j - \epsilon)|V_{i_0}|, (\pi_j + \epsilon)|V_{i_0}|\Big) \: \forall j \in Z\Big\},
\end{equation*}
with $\epsilon > 0$ to be defined. We will show that the restricted MLE achieves exact recovery on $V_{i_0}$ with high probability. Then, since the MAP estimator is optimal, we obtain that the MAP estimator achieves exact recovery on $V_{i_0}$ with high probability as well.

\begin{lemma}
\label{lem:restricted_MLE_lemma}
    Let $G \sim \text{GHCM}(\lambda, n, r, \pi, P(y), d)$ and let $\bar{\sigma}$ be the restricted MLE. Suppose that Assumptions \ref{assump:regularity_condition}, \ref{assump:bounded_log_likelihood}, and \ref{assump:distinctness} hold. Then, $\bar{\sigma}$ achieves exact recovery on $V_{i_0}$ with high probability.
\end{lemma}

To prove Lemma \ref{lem:restricted_MLE_lemma}, we upper-bound the probability there exists an incorrect labeling $\sigma \in X_0^*(\epsilon)$ with a higher posterior likelihood of the observed graph than the ground-truth labeling $\sigma^*$; since $\sigma^* \in X_0^*(\epsilon)$ with high probability by \cite[Lemma D.3]{Gaudio+2025}, this is the error probability of the restricted MLE. Formally, define the posterior log-likelihood of the observed graph relative to the labeling $\sigma$ as
\begin{equation*}
    \ell_0(G, \sigma) := \sum_{u \in V_{i_0}} \sum_{v \in V_{i_0}, v \neq u} \log(\overline{p}_{\sigma(u)\sigma(v)}(x_{uv}; \|u - v\|)).
\end{equation*}
Then, we upper-bound
\begin{equation}
\label{eq:restricted_MLE_fail_prob}
    \P\bigg( \bigcup_{\sigma \in X_0^*(\epsilon)} \{\ell_0(G, \sigma) - \ell_0(G, \sigma^*) \geq 0\} \bigg).
\end{equation}

We divide the analysis into two cases, depending on the \textit{discrepancy} between the incorrect labeling $\sigma$ and the ground-truth labeling $\sigma^*$. We use the definition of  \cite{Gaudio+2025}.
\begin{definition}[Discrepancy]
    Let $d_H(\sigma, \sigma')$ be the Hamming distance between two labelings $\sigma, \sigma' : V_{i_0} \to Z$, which is defined
    \begin{equation*}
        d_H(\sigma, \sigma') = \sum_{u \in V_{i_0}} \1(\sigma(u) \neq \sigma'(u)).
    \end{equation*}
    Then, the discrepancy between two $\sigma$ and $\sigma'$ is defined as
    \begin{equation*}
        \text{DISC}(\sigma, \sigma') = \min_{\omega \in \Omega_{\pi, P}} \{d_H(\omega \circ \sigma, \sigma')\},
    \end{equation*}
    which is the number of vertices on which $\sigma$ and $\sigma'$ differ, up to a permissible relabeling. We note that a labeling $\sigma$ is correct if and only if $\text{DISC}(\sigma, \sigma^*) = 0$.
\end{definition}

Let $\eta$ be as defined in Assumption \ref{assump:regularity_condition} and $\epsilon_0$ be defined as in Algorithm \ref{alg:exact_recovery}. Let $\pi_\text{min} = \min_{i \in Z} \pi_i$. Throughout this section, we fix $t \in (0,1)$ to be an arbitrary constant. Fix a vertex $u \in V_{i_0}$, and let $D$ be the distance between $u$ and another vertex placed uniformly at random in the initial block $B_{i_1}$. Let $\Phi_1 < 1$ be such that $\phi_t(\overline{P}_{ij}(D), \overline{P}_{ab}(D)) < \Phi_1)$ for all $i,j,a,b \in Z$ with $P_{ij} \not \equiv P_{ab}$ with probability at least $3/4$. The existence of $\Phi_1$ is guaranteed by Corollary \ref{cor:distinguishing_vertices}, and note that $\Phi_1$ can be chosen uniformly for all $u \in V_{i_0}$. We choose positive constants $c$ and $\epsilon$ such that
\begin{equation}
\label{eq:def_c_epsilon}
    c \leq \frac{\epsilon_0 \log(1/\Phi_1)}{4 \eta} \quad \text{and} \quad \epsilon \leq \frac{1}{3}\min\bigg\{\pi_\text{min}, \min_{i,j \in Z: \pi_i \neq \pi_j} \{|\pi_i - \pi_j|\}, \frac{c}{(k-1)\epsilon_0}\bigg\}.
\end{equation}
Also let $\Phi_!' < 1$ be such that $\phi_t(\overline{P}_{ij}(D), \overline{P}_{ab}(D)) < \Phi_1'$ for all $i,j,a,b \in Z$ with $P_{ij} \not \equiv P_{ab}$ with probability at least $1 - \frac{\epsilon}{4k}$.

The two cases of low and high discrepancy are handled in the following propositions, in parallel to \cite[Propositions D.5 and D.6]{Gaudio+2025}.
\begin{prop}[Low Discrepancy]
\label{prop:low_discrepancy}
    Suppose that Assumptions \ref{assump:bounded_log_likelihood} and \ref{assump:distinctness} hold. Then, with $\epsilon$ defined in \eqref{eq:def_c_epsilon}, there exists $c \in (0, 1)$ satisfying \eqref{eq:def_c_epsilon} such that with high probability
    \begin{equation*}
        \forall \sigma \in X_0^*(\epsilon) \text{ such that } 0 < \mathrm{DISC}(\sigma, \sigma^*) < c \log n, \text{ we have } \ell_0(G, \sigma) < \ell_0(G, \sigma^*).
    \end{equation*}
\end{prop}

\begin{prop}[High Discrepancy]
\label{prop:high_discrepancy}
    Suppose that Assumption \ref{assump:regularity_condition} holds. Fix any $c \in (0, 1)$. Then, with $\epsilon$ defined in \eqref{eq:def_c_epsilon}, with high probability
    \begin{equation*}
        \forall \sigma \in X_0^*(\epsilon) \text{ such that } \mathrm{DISC}(\sigma, \sigma^*) \geq c \log n, \text{ we have } \ell_0(G, \sigma) < \ell_0(G, \sigma^*).
    \end{equation*}
\end{prop}

\begin{proof}[Proof of Lemma \ref{lem:restricted_MLE_lemma}]
    Fix $c$ and $\epsilon$ satisfying \eqref{eq:def_c_epsilon}. Combining Propositions \ref{prop:low_discrepancy} and \ref{prop:high_discrepancy} yields that with high probability,
    \begin{equation*}
        \forall \sigma \in X_0^*(\epsilon) \text{ with } \mathrm{DISC}(\sigma, \sigma^*) \neq 0, \text{ we have } \ell_0(G, \sigma) < \ell_0(G, \sigma^*).
    \end{equation*}
    Since $\sigma^* \in X_0^*(\epsilon)$ with high probability due to \cite[Lemma D.3]{Gaudio+2025}, the posterior likelihood of $\sigma^*$ exceeds the posterior likelihood of any incorrect labeling with high probability. Therefore, the restricted MLE achieves exact recovery with high probability.
\end{proof}

We now prove Propositions \ref{prop:low_discrepancy} and \ref{prop:high_discrepancy}.

\begin{proof}[Proof of Proposition \ref{prop:low_discrepancy}]
    Fix $\sigma$ such that $\text{DISC}(\sigma, \sigma^*) > 0$. Let $\omega \in \Omega_{\pi, P}$ be the permissible relabeling such that $d_H(\omega \circ \sigma, \sigma^*) = \text{DISC}(\sigma, \sigma^*)$. Gaudio et al. \cite{Gaudio+2025} observe that $\ell_0(G, \sigma) = \ell_0(G, \omega \circ \sigma)$; thus we can assume without loss of generality that $\omega$ is the identity and $d_H(\sigma, \sigma^*) = \text{DISC}(\sigma, \sigma^*)$. Since we want to upper-bound \eqref{eq:restricted_MLE_fail_prob}, we first compute the difference of the posterior likelihoods, which yields
    \begin{align}
    \label{eq:MAP_log_likelihood_difference}
        &\ell_0(G, \sigma) - \ell_0(G, \sigma^*) = \sum_{u \in V_{i_0}} \sum_{v \in V_{i_0}, v \neq u} \log \bigg(\frac{\overline{p}_{\sigma(u)\sigma(v)}(x_{uv}; \|u - v\|)}{\overline{p}_{\sigma^*(u)\sigma^*(v)}(x_{uv}; \|u - v\|)} \bigg) \notag \\
        &\quad= \sum_{u \in V_{i_0}} \sum_{v \in V_{i_0}, v \neq u} \log \bigg(\frac{\overline{p}_{\sigma(u)\sigma^*(v)}(x_{uv}; \|u - v\|)}{\overline{p}_{\sigma^*(u)\sigma^*(v)}(x_{uv}; \|u - v\|)} \bigg) 
        + \sum_{u \in V_{i_0}} \sum_{v \in V_{i_0}, v \neq u} \log \bigg(\frac{\overline{p}_{\sigma(u)\sigma(v)}(x_{uv}; \|u - v\|)}{\overline{p}_{\sigma(u)\sigma^*(v)}(x_{uv}; \|u - v\|)} \bigg) \notag \\
        &\quad= \sum_{\substack{u \in V_{i_0} \\ \sigma(u) \neq \sigma^*(u)}} \sum_{v \in V_{i_0}, v \neq u} \log \bigg(\frac{\overline{p}_{\sigma(u)\sigma^*(v)}(x_{uv}; \|u - v\|)}{\overline{p}_{\sigma^*(u)\sigma^*(v)}(x_{uv}; \|u - v\|)} \bigg) 
        + \sum_{\substack{v \in V_{i_0} \\ \sigma(v) \neq \sigma^*(v)}} \sum_{u \in V_{i_0}, u \neq v} \log \bigg(\frac{\overline{p}_{\sigma(u)\sigma(v)}(x_{uv}; \|u - v\|)}{\overline{p}_{\sigma(u)\sigma^*(v)}(x_{uv}; \|u - v\|)} \bigg).
    \end{align}
    We will show that with high probability, this difference is negative for all low-discrepancy labelings. We consider each summation of \eqref{eq:MAP_log_likelihood_difference} separately. Let $A(\sigma)$ denote the first summation and $B(\sigma)$ denote the second summation.

    For $A(\sigma)$, we consider the contribution to the sum from a fixed $u \in V_{i_0}$ where $\sigma(u) \neq \sigma^*(u)$. Fix $0 < c_1 \leq \epsilon_0 \log(1 / \Phi_1) / 4$. Let $\sigma^*(u) = a$, $\sigma(u) = b$. Then, the Chernoff bound yields that for any $t > 0$, conditioned on the locations of vertices given by $L_{i_0}$, 
    \begin{align*}
        &\P\bigg(\sum_{v \in V_{i_0}, v \neq u} \log \bigg(\frac{\overline{p}_{b,\sigma^*(v)}(x_{uv}; \|u - v\|)}{\overline{p}_{a,\sigma^*(v)}(x_{uv}; \|u - v\|)} \bigg) \geq -c_1 \log n \: \Big| \: L_{i_0}\bigg) \\
        &\quad\leq \E\bigg[ \exp\bigg(t \sum_{v \in V_{i_0}, v \neq u} \log \bigg(\frac{\overline{p}_{b,\sigma^*(v)}(x_{uv}; \|u - v\|)}{\overline{p}_{a,\sigma^*(v)}(x_{uv}; \|u - v\|)} \bigg) \bigg) \: \Big| \: L_{i_0}  \bigg] \cdot \exp(tc_1 \log n) \\
        &\quad= \prod_{v \in V_{i_0}, v \neq u} \E\bigg[ \bigg(\frac{\overline{p}_{b,\sigma^*(v)}(x_{uv}; \|u - v\|)}{\overline{p}_{a,\sigma^*(v)}(x_{uv}; \|u - v\|)}\bigg)^t \: \Big| \: L_{i_0} \bigg] \cdot n^{tc_1} \\
        &\quad= n^{tc_1} \prod_{v \in V_{i_0}, v \neq u} \phi_t(\overline{P}_{b,\sigma^*(v)}(\|u - v\|), \overline{P}_{a,\sigma^*(v)}(\|u - v\|)).
    \end{align*}

    By definition of $\Phi_1$ and using Assumption \ref{assump:distinctness}, we have that the number of vertices $v \in V_{i_0}$ for which $\phi_t(\overline{P}_{b,\sigma^*(v)}(\|u - v\|), \overline{P}_{a,\sigma^*(v)}(\|u - v\|)) < \Phi_1$ dominates a binomial random variable with parameters $(\epsilon_0 \log n, 3/4)$. Therefore, there are at least $(\epsilon_0 \log n)/2$ such vertices with probability at least $1 - \exp\left(-(2 \epsilon_0 \log n)/{9} \right) = 1 - n^{-2\epsilon_0/9}$ by the binomial Chernoff bound. Therefore, removing the conditioning on $L_{i_0}$ yields
    \begin{align*}
    \P\bigg(\sum_{v \in V_{i_0}, v \neq u} \log \bigg(\frac{\overline{p}_{b,\sigma^*(v)}(x_{uv}; \|u - v\|)}{\overline{p}_{a,\sigma^*(v)}(x_{uv}; \|u - v\|)} \bigg) \geq -c_1 \log n \bigg) 
        &\leq n^{tc_1} \Phi_1^{(\epsilon_0 \log n) / 2} + n^{-2\epsilon_0/9} \notag \\
        &\leq n^{\epsilon_0 \log(\Phi_1)/2 + c_1} + n^{-2\epsilon_0/9} \notag \\
        &\leq n^{\epsilon_0 \log(\Phi_1)/4} + n^{-2\epsilon_0/9},
    \end{align*} 
    where the last inequality follows from the definition of $c_1$. Note that this result holds for any vertex $u \in V_{i_0}$ and community $b \neq \sigma^*(u)$. Therefore, we can apply the union bound to obtain
    \begin{align*}
        &\P\Bigg( \bigcap_{u \in V_{i_0}} \bigcap_{b \neq \sigma^*(u)} \bigg\{ \sum_{v \in V_{i_0}, v \neq u} \log \bigg(\frac{\overline{p}_{b,\sigma^*(v)}(x_{uv}; \|u - v\|)}{\overline{p}_{a,\sigma^*(v)}(x_{uv}; \|u - v\|)} \bigg) \leq -c_1 \log n \bigg\} \Bigg) \\
        &\quad \geq 1 - k (\epsilon_0 \log n) \cdot \left(n^{\epsilon_0 \log(\Phi_1) / 4}  + n^{-2\epsilon_0/9}\right) \\
        &\quad= 1 - o(1).
    \end{align*}
    Finally, observe that the event  
    \begin{equation*}
        \bigcap_{u \in V_{i_0}} \bigcap_{b \neq \sigma^*(u)} \bigg\{ \sum_{v \in V_{i_0}, v \neq u} \log \bigg(\frac{\overline{p}_{b,\sigma^*(v)}(x_{uv}; \|u - v\|)}{\overline{p}_{a,\sigma^*(v)}(x_{uv}; \|u - v\|)} \bigg) \leq -c_1 \log n \bigg\}
    \end{equation*}
    implies that $A(\sigma) \leq -d_H(\sigma, \sigma^*) c_1 \log n$ for any labeling $\sigma$ because there are $d_H(\sigma, \sigma^*)$ vertices $u \in V_{i_0}$ for which $\sigma(u) \neq \sigma^*(u)$. Therefore, we obtain that
    \begin{equation*}
        \P\bigg(\bigcap_{\sigma:V_{i_0} \to Z} \{A(\sigma) \leq -d_H(\sigma, \sigma^*) c_1 \log n \} \bigg) = 1 - o(1),
    \end{equation*}
    which gives a high-probability upper bound for $A(\sigma)$ over all labelings.

    Now, we consider the second summation $B(\sigma)$. Following the same calculations as in \cite{Gaudio+2025}, we obtain
    \begin{equation}
    \label{eq:B_sigma_expression}
    \begin{aligned}
        B(\sigma) &= \sum_{\substack{v \in V_{i_0} \\ \sigma(v) \neq \sigma^*(v)}} \sum_{u \in V_{i_0}, u \neq v}  \log \bigg(\frac{\overline{p}_{\sigma^*(u)\sigma(v)}(x_{uv}; \|u - v\|)}{\overline{p}_{\sigma^*(u)\sigma^*(v)}(x_{uv}; \|u - v\|)} \bigg) \\
        &\qquad+ \sum_{\substack{v \in V_{i_0} \\ \sigma(v) \neq \sigma^*(v)}} \sum_{\substack{u \in V_{i_0}, u \neq v \\ \sigma(u) \neq \sigma^*(u)}} \log \bigg(\frac{\overline{p}_{\sigma(u)\sigma(v)}(x_{uv}; \|u - v\|)}{\overline{p}_{\sigma(u)\sigma^*(v)}(x_{uv}; \|u - v\|)} \bigg) + \log \bigg(\frac{\overline{p}_{\sigma^*(u)\sigma^*(v)}(x_{uv}; \|u - v\|)}{\overline{p}_{\sigma^*(u)\sigma(v)}(x_{uv}; \|u - v\|)} \bigg) 
    \end{aligned}
    \end{equation}
    Let $B_1(\sigma)$ denote the first summation and $B_2(\sigma)$ denote the second summation in  \eqref{eq:B_sigma_expression}. Since $B_1(\sigma) = A(\sigma)$, we obtain that
    \begin{equation*}
        \P\bigg(\bigcap_{\sigma:V_{i_0} \to Z} \{B_1(\sigma) \leq -d_H(\sigma, \sigma^*) c_1 \log n \} \bigg) = 1 - o(1)
    \end{equation*}
    Then, by Assumption \ref{assump:bounded_log_likelihood}, we obtain that $B_2(\sigma) \leq 2 \eta d_H(\sigma, \sigma^*)^2$ because each term is at most $2 \eta$ and there are $d_H(\sigma, \sigma^*)^2$ total terms. 

    Therefore, with high probability we have $A(\sigma) \leq -d_H(\sigma, \sigma^*)c_1 \log n$, $B_1(\sigma) \leq -d_H(\sigma, \sigma^*)c_1 \log n$, and $B_2(\sigma) \leq 2\eta d_H(\sigma, \sigma^*)^2$ for all $\sigma$, which implies that
    \begin{equation}
    \label{eq:MAP_log_likelihood_difference_final}
        \ell_0(G, \sigma) - \ell_0(G, \sigma^*) \leq -2d_H(\sigma, \sigma^*)c_1 \log n + 2 \eta d_H(\sigma, \sigma^*)^2
    \end{equation}
    for all $\sigma$ with high probability. When $0 < \text{DISC}(\sigma, \sigma^*) < c \log n$ for any $c \leq c_1/\eta$, we obtain that \eqref{eq:MAP_log_likelihood_difference_final} is negative, which means that with high probability,
    \begin{equation*}
        \forall \sigma \text{ such that } 0 < \text{DISC}(\sigma, \sigma^*) < c \log n, \text{ we have } \ell_0(G,\sigma) - \ell_0(G, \sigma^*) < 0.
    \end{equation*}
\end{proof}

\begin{proof}[Proof of Proposition \ref{prop:high_discrepancy}]
Fix $c \in (0, 1)$ and $\sigma \in X_0^*(\epsilon)$ such that $\text{DISC}(\sigma, \sigma^*) \geq c \log n$. We will bound the probability that $\ell_0(G, \sigma) - \ell_0(G, \sigma^*) \geq 0$ (i.e. that $\sigma$ causes the restricted MLE to fail), then apply a union bound over all high-discrepancy labelings. Using the Chernoff bound, for any $t > 0$, conditioned on the locations of vertices given by $L_{i_0}$,
\begin{align}
    \P(\ell_0(G, \sigma) - \ell_0(G, \sigma^*) > 0 \mid L_{i_0}) &\leq \E\big[\exp\big(t(\ell_0(G, \sigma) - \ell_0(G, \sigma^*))\big) \mid L_{i_0}\big] \nonumber \\
    &= \E\bigg[\prod_{u \in V_{i_0}} \prod_{v \in V_{i_0}, v \neq u} \bigg(\frac{p_{\sigma(u)\sigma(v)}(x_{uv};\|u-v\|)}{p_{\sigma^*(u)\sigma^*(v)}(x_{uv};\|u-v\|)}\bigg)^t \: \Big| \: L_{i_0}\bigg] \nonumber \\
    &= \prod_{u \in V_{i_0}} \prod_{v \in V_{i_0}, v \neq u} \E\bigg[ \bigg(\frac{p_{\sigma(u)\sigma(v)}(x_{uv};\|u-v\|)}{p_{\sigma^*(u)\sigma^*(v)}(x_{uv};\|u-v\|)}\bigg)^t \: \Big| \: L_{i_0} \bigg] \nonumber \\
    &=  \prod_{u \in V_{i_0}} \prod_{v \in V_{i_0}, v \neq u} \phi_t(\overline{P}_{\sigma(u)\sigma(v)}(\|u-v\|), \overline{P}_{\sigma^*(u)\sigma^*(v)}(\|u-v\|)). \label{eq:likelihood-Chernoff}
\end{align}
Observe that for every $u \in V_{i_0}$, the number of vertices $v \in V_{i_0}$ such that 
\[\phi_t(\overline{P}_{ij}(\|u-v\|), \overline{P}_{i'j'}(\|u-v\|)) > \Phi_1'\]
for some $i,j,i',j'$ with $P_{ij} \not \equiv P_{i'j'}$ is at most $\frac{\epsilon}{2k} \cdot \epsilon_0 \log n$ with probability $1 - n^{-\Omega(1)}$ by a binomial Chernoff bound, using Corollary \ref{cor:distinguishing_vertices}. Taking a union bound, we see that the same holds simultaneously for all $u \in V_{i_0}$ with probability $1-o(1)$, and we refer to this event as $\{L_{i_0} \in \mathcal{L}_{\text{good}}\}$. 

Now, let $N_{ij} := \{u \in V_{i_0}: x^{\star}(u) = i, x(u) = j\}$ and let $n_{ij} := |N_{ij}|$ for any $i,j \in Z$. 
We split the analysis into two cases. 

Case 1: \textit{There exist $a,b,b'$ be such that $n_{ab}, n_{ab'} \geq \epsilon \epsilon_0 \log (n)/k$.} Take $z \in Z$ so that $P_{bz} \not \equiv P_{b'z}$, which exists due to Assumption \ref{assump:regularity_condition} as otherwise $b,b'$ would be statistically indistinguishable. Let $y \in Z$ satisfy $n_{yz} \geq |u : \sigma(u) = z|/k$. Using $\sigma \in X_0^{\star}(\epsilon)$ and the assumption $\epsilon \leq \pi_{\text{min}}/3$, we get $n_{yz} \geq \frac{\epsilon \epsilon_0 \log n}{k}$.

Since $P_{bz} \not \equiv P_{b'z}$, either $P_{ay} \not \equiv P_{bz}$ or $P_{ay} \not \equiv P_{b'z}$. Observe that for every $u \in V_{i_0}$, either the number of vertices $v \in V_{i_0}$ such that 
\[\phi_t(\overline{P}_{bz}(\|u-v\|), \overline{P}_{ay}(\|u-v\|)) > \Phi_1'\]
is at most $\frac{\epsilon}{2k} \cdot \epsilon_0 \log n$,
or the number of vertices $v \in V_{i_0}$ such that
\[\phi_t(\overline{P}_{b'z}(\|u-v\|), \overline{P}_{ay}(\|u-v\|)) > \Phi_1'\]
is at most $\frac{\epsilon}{2k} \cdot \epsilon_0 \log n$ whenever $L_{i_0} \in \mathcal{L}_{\text{good}}$. In that situation, we claim that \eqref{eq:likelihood-Chernoff} is upper-bounded by $(\Phi_1')^{\frac{1}{2}\left(\frac{\epsilon}{k} \cdot \epsilon_0 \log n \right)^2}$. To see this, suppose $P_{ay} \not \equiv P_{bz}$. Recall that $n_{ab} \geq \epsilon \epsilon_0 \log(n)/k$. For every $u \in N_{ab}$, the number of vertices $v \in N_{yz}$ such that \[\phi_t(\overline{P}_{bz}(\|u-v\|), \overline{P}_{ay}(\|u-v\|)) \leq \Phi_1'\]
is at least $n_{yz} - \frac{\epsilon}{2k} \cdot \epsilon_0 \log n \geq \frac{\epsilon}{2k} \cdot \epsilon_0 \log n$. The same reasoning applies if $P_{ay} \not \equiv P_{b'z}$.

Case 2: \textit{For every community $i \in Z$, there is at most one $j \in Z$ such that $n_{ij} \geq \epsilon \epsilon_0 \log (n)/k$.} Let $\omega : Z \to Z$ be a mapping such that $n_{i \omega(i)} \geq \epsilon \epsilon_0 \log (n) /k$. We note that $\omega$ is unique, as established in \cite[Proposition D.6]{Gaudio+2025}. Following the steps of \cite[Proposition D.6]{Gaudio+2025}, using the assumptions on $\epsilon, \epsilon_0, c$, we can show that there must exist $(a,b)$ for which $P_{ab} \not \equiv P_{\omega(a), \omega(b)}$. Similarly to Case 1, we obtain that \eqref{eq:likelihood-Chernoff} is upper-bounded by $(\Phi_1')^{\frac{1}{2}\left(\frac{\epsilon}{k} \cdot \epsilon_0 \log n \right)^2}$ whenever $L_{i_0} \in \mathcal{L}_{\text{good}}$.

Finally, putting both cases together yield a uniform constant $C \in (0,1)$ such that
\[\mathbb{P}(\ell_0(G,\sigma) -\ell_0(G, \sigma^{\star}) > 0 \mid L_{i_0} \in \mathcal{L}_{\text{good}}) \leq C^{\log^2 n}\]
for all $\sigma \in X_0^{\star}(\epsilon)$ satisfying $\text{DISC}(\sigma, \sigma^{\star}) > c \log n$. Taking a union bound over such labelings, of which there are at most $k^{\epsilon_0 \log n}$, yields
\[\mathbb{P}\Bigg(\bigcup_{\underset{ \text{DISC}(\sigma, \sigma^{\star}) > c \log n}{\sigma \in X_0^{\star}(\epsilon) :}} \left\{\ell_0(G,\sigma) - \ell_0(G, \sigma^{\star} > 0 \right\} \: \bigg| \: L_{i_0} \in \mathcal{L}_{\text{good}}\Bigg) \leq k^{\epsilon_0 \log n} \cdot C^{\log^2 n}. \]
Finally, since $\mathbb{P}(L_{i_0} \in \mathcal{L}_{\text{good}}) = 1-o(1)$
, it follows that
\[\mathbb{P}\Bigg(\bigcup_{\underset{ \text{DISC}(\sigma, \sigma^{\star}) > c \log n}{\sigma \in X_0^{\star}(\epsilon) :}} \left\{\ell_0(G,\sigma) - \ell_0(G, \sigma^{\star} > 0 \right\} \Bigg) \leq k^{\epsilon_0 \log n} \cdot C^{\log^2 n} + o(1) = o(1). \]
\end{proof}

\subsection{Propagate}
\label{subappx:propagate_proof}
In this section, we show that using \texttt{Propagate} to label all remaining $\delta$-occupied blocks yields almost-exact recovery. For any block $V_i$, let $C_i$ denote the location, true label, and estimated label of each vertex in $V_i$, which we call the configuration of $V_i$. We first provide a result showing that each parent block has sufficient ``useful'' vertices. 

\begin{definition}
    Fix a vertex $v \in V \setminus V_{i_0}$, letting $v \in V_i$ without loss of generality, and suppose that $V_{p(i)}$ is the parent block. Let $j = \argmax_{a \in Z} |\{ u \in V_{p(i)}: \hat{\sigma}(u) = a \}|$ be the largest community in $V_{p(i)}$ as determined by $\hat{\sigma}$. We say that $v$ is $\Phi$-\textit{distinguished} by $V_{p(i)}$ is there exist at least $(\delta \log n)/(2k)$ vertices $u \in V_{p(i)}$ with $\hat{\sigma}(u) = j$ such that
    \begin{equation*}
        \phi_t(\overline{P}_{ab}(\|u-v\|), \overline{P}_{a'b'}(\|u-v\|)) < \Phi
    \end{equation*}
    for all $a,b,a',b' \in Z$ with $P_{ab} \not \equiv P_{a'b'}$.
\end{definition}

\begin{lemma}
\label{lem:propagate_distinguishing_vertices}
    There exists $0 < \Phi < 1$ such that all vertices $v \in V \setminus V_{i_0}$ are $\Phi$-distinguished by their respective parent block with high probability.
\end{lemma}
\begin{proof}
    Fix a vertex $v \in V \setminus V_{i_0}$ with parent block $V_{p(i)}$. Let $D$ be the distance from $v$ to a vertex placed uniformly at random in the parent block $B_{p(i)}$. Fix $\alpha > 0$, and observe that there exists $0 < \Phi < 1$ such that 
    \begin{equation*}
        \phi_t(\overline{P}_{ab}(D), \overline{P}_{a'b'}(D)) \geq \Phi
    \end{equation*}
    for some $a,b,a',b' \in Z$ with $P_{ab} \not \equiv P_{a'b'}$ with probability at most $\alpha$ due to Corollary \ref{cor:distinguishing_vertices}. Note that $\Phi$ is uniform over all $v \in V$. Now, let $j = \argmax_{a \in Z} |\{ u \in V_{p(i)}: \hat{\sigma}(u) = a \}|$ and observe that the number of vertices $u \in V_{p(i)}$ for which 
    \begin{equation*}
        \phi_t(\overline{P}_{ab}(\|u-v\|), \overline{P}_{a'b'}(\|u-v\|)) \geq \Phi
    \end{equation*}
    for some $a,b,a',b' \in Z$ with $P_{ab} \not \equiv P_{a'b'}$ is stochastically dominated by a $\text{Bin}(\delta\log n, \alpha)$ random variable. To show that $v$ is $\Phi$-distinguished by $V_{p(i)}$, we must show that the number of such vertices is at most $(\delta \log n) / (2k)$. Hence, we apply the Chernoff bound (Lemma \ref{lem:binomial_Chernoff}), which yields
    \begin{align}
    \label{eq:lack_distinguishing_vertices_prob}
        \P\left(\text{Bin}\left(\delta \log n, \alpha \right) \geq \frac{\delta \log n}{2k} \right) 
        \notag &\leq \left( \frac{e^{(1/(2\alpha k)) - 1}}{(1/(2\alpha k))^{(1/(2\alpha k))}}\right)^{(\alpha \delta \log n)} \notag \\
        &= \exp\left(\log\left(\frac{e^{(1/(2\alpha k)) - 1}}{(1/(2\alpha k))^{(1/(2\alpha k))}}\right) \alpha \delta \log n \right) \notag \\
        &= \exp\left(\left(\frac{1}{2\alpha k} - 1 - \frac{1}{2\alpha k}\log\left(\frac{1}{2\alpha k}\right)\right) \alpha \delta \log n \right) \notag \\
        &= \exp\left(\left(1 - 2\alpha k + \log(2\alpha k)\right) \frac{\delta \log n}{2k} \right) \notag \\
        &= n^{\delta (1 - 2\alpha k + \log(2\alpha k)) / (2k) }.
    \end{align}
    Choosing $\alpha > 0$ small enough, we obtain that \eqref{eq:lack_distinguishing_vertices_prob} is $o(1/n)$. Therefore, the number of vertices $u \in V_{p(i)}$ for which 
    \begin{equation*}
        \phi_t(\overline{P}_{ab}(\|u-v\|), \overline{P}_{a'b'}(\|u-v\|)) \geq \Phi
    \end{equation*} 
     for some $a,b,a',b' \in Z$ with $P_{ab} \not \equiv P_{a'b'}$ is at most $(\delta \log n) / (2k)$ with probability $o(1/n)$. Since there are at least $(\delta \log n) / k$ vertices $u \in V_{p(i)}$ such that $\hat{\sigma}(u) = j$, we obtain that $v$ is $\Phi$-distinguished by $V_{p(i)}$ with probability $1 - o(1/n)$.

    We would now like to take a union bound over all vertices $v \in V$, of which there are $\Theta(n)$ many, to obtain that all vertices are $\Phi$-distinguished by their parent block with high probability. However, since the number of vertices is random, we use a union bound over a high probability upper bound on the number of vertices. Let $c > \lambda$, and note that $\P(|V| \geq cn) = o(1)$ by the Poisson Chernoff bound (Lemma \ref{lem:poisson_Chernoff}). Enumerate the vertices as $v_1, v_2, \ldots, v_{|V|}$. Let $E_i := \{|V| \geq i\}$ be the event that $v_i$ exists, and $F_i$ be the event that $v_i$ is $\Phi$-distinguished by its parent block. Noting that $\{E_i^c \cup F_i\}$ is the event that the $i^\text{th}$ vertex either does not exist or is $\Phi$-distinguished, observe that
    \begin{align*}
        \P(\bigcap_{i=1}^\infty \{E_i^c \cup F_i\}) &= 1 - \P(\bigcup_{i=1}^\infty \{E_i \cap F_i^c\} \\
        &\geq 1 - \sum_{i=1}^{cn} \P(F_i^c) - \P(\bigcup_{i=cn+1}^\infty E_i) \\
        &= 1 - \sum_{i=1}^{cn}o(1/n) - \P(|V| > cn) \\
        &= 1 - o(1)
    \end{align*}
    Therefore, all vertices $v \in V$ are $\Phi$-distinguished by their parent block with high probability.
\end{proof}

We now bound the probability that $\hat{\sigma}$ makes more than a constant $M$ mistakes in any given block using \texttt{Propagate}, adapting the proof from \cite{Gaudio+2025}. Let $\omega^* \in \Omega_{\pi, P}$ be the permissible relabeling corresponding to the estimated labels of the initial block, that is $\hat{\sigma}(v) = \omega^* \circ \sigma^*(v)$ for all $v \in V_{i_0}$. We condition on the event that $\omega^*$ exists, which occurs with high probability due to Proposition \ref{prop:MAP}.

\begin{prop}
\label{prop:propagation_error_block}
Let $\delta > 0$ and $\Phi_2$ be such that Lemma \ref{lem:propagate_distinguishing_vertices} is satisfied. Suppose that Assumptions \ref{assump:bounded_log_likelihood} and \ref{assump:distinctness} hold. Define $c_2 := \delta \log(1/\Phi_2) / (3k)$, $\eta_2 := e^{\eta M}$, and $M := 5/(4c_2)$. Let $B_i$ be a $\delta$-occupied block with vertices $V_i$, and let $B_{p(i)}$ with vertices $V_{p(i)}$ be its parent block. Suppose that $V_{p(i)}$ has configuration $C_{p(i)}$ such that $\hat{\sigma}$ makes at most $M$ mistakes on $V_{p(i)}$ with permissible relabeling $\omega^*$ and that all vertices $v \in V_i$ are $\Phi_2$-distinguished by $V_{p(i)}$. Suppose that $\hat{\sigma}$ labels $V_i$ using Propagate. Then, the probability that $\hat{\sigma}$ labels any given vertex $v \in V_i$ incorrectly is
\begin{equation*}
    \P(\hat{\sigma}(v) \neq \omega^* \circ \sigma^*(v) \mid C_{p(i)}) \leq \eta_2 n^{-c_2}.
\end{equation*}
Furthermore, suppose that $|V_i| < \Delta \log n$ for some constant $\Delta$, and let $\eta_3 := e^M(\eta_2 \Delta / M)^M$. Then, the probability that $\hat{\sigma}$ makes more than $M$ mistakes on $V_i$ is
\begin{equation*}
    \P\bigg( | \{v \in V_i: \hat{\sigma}(v) \neq \omega^* \circ \sigma^*(v)\} | > M \: \Big| \: C_{p(i)}, |V_i| \leq \Delta \log n \bigg) \leq \eta_3 n^{-9/8}.
\end{equation*} 
\end{prop}
\begin{proof}
Let $j := \argmax_{a \in Z} |\{u \in S: \hat{\sigma}_S(u) = a\}|$ be the largest community in $V_{p(i)}$ according to $\hat{\sigma}$. Suppose that $\omega^* \circ \sigma^*(v) = \ell$. We must upper bound the probability that $\hat{\sigma}$ labels $v$ as anything other than $\ell$. Let $s \neq \ell \in Z$, we start by upper bounding the probability that $\hat{\sigma}$ labels $v$ as $s$, which is 
\begin{equation*}
    \P(\hat{\sigma}(v) = s \mid C_{p(i)}, \omega^* \circ \sigma^*(v) = \ell).
\end{equation*}
Observe that this only occurs when the likelihood of $v$ having label $s$ exceeds the likelihood of $v$ having label $\ell$, with respect to the vertices $u \in V_{p(i)}$ with $\hat{\sigma}(u) = j$. Hence, using the Chernoff bound, we obtain that for any $t > 0$,
\begin{align}
\label{eq:propagate_single_vertex_error_1}
    &\P(\hat{\sigma}(v) = s \mid C_{p(i)}, \omega^* \circ \sigma^*(v) = \ell) \notag \\
    &\qquad\leq \P\bigg(\sum_{u \in V_{p(i)}: \hat{\sigma}(u) = j} \log(\overline{p}_{js}(x_{uv};\|u-v\|) - \log(\overline{p}_{j\ell}(x_{uv};\|u-v\|) \geq 0 \: \Big| \: C_{p(i)}, \omega^* \circ \sigma^*(v) = \ell \bigg) \notag \\
    &\qquad= \P\bigg(\sum_{u \in V_{p(i)}: \hat{\sigma}(u) = j} \log\Big(\frac{\overline{p}_{js}(x_{uv};\|u-v\|)}{\overline{p}_{j\ell}(x_{uv};\|u-v\|)}\Big) \geq 0 \: \Big| \: C_{p(i)}, \omega^* \circ \sigma^*(v) = \ell \bigg) \notag \\
    &\qquad\leq \E\bigg[\exp\bigg(t\sum_{u \in V_{p(i)}: \hat{\sigma}(u) = j} \log\Big(\frac{\overline{p}_{js}(x_{uv};\|u-v\|)}{\overline{p}_{j\ell}(x_{uv};\|u-v\|)}\Big)\bigg) \: \Big| \: C_{p(i)}, \omega^* \circ \sigma^*(v) = \ell \bigg] \notag \\
    &\qquad= \prod_{u \in V_{p(i)}:\hat{\sigma}(u)=j} \E\bigg[ \Big(\frac{\overline{p}_{js}(x_{uv};\|u-v\|)}{\overline{p}_{j\ell}(x_{uv};\|u-v\|)}\Big)^t \: \Big| \: C_{p(i)}, \omega^* \circ \sigma^*(v) = \ell \bigg].
\end{align}
We split the product into two factors, depending on whether $\hat{\sigma}(u)$ is correct or not. Let
\begin{equation*}
    R_+(v) := \{u \in V_{p(i)}: \hat{\sigma}(u) = j, \omega^* \circ \sigma^*(u) = j\}
\end{equation*}
and 
\begin{equation*}
    R_-(v) := \{u \in V_{p(i)}: \hat{\sigma}(u) = j, \omega^* \circ \sigma^*(u) \neq j\}.
\end{equation*}
That is, $R_+(v)$ is the set of vertices which are labeled correctly and $R_-(v)$ is the set of vertices which are labeled incorrectly. Then, the product in \eqref{eq:propagate_single_vertex_error_1} becomes
\begin{align}
\label{eq:propagate_single_vertex_error_2}
    &\P(\hat{\sigma}(v) = s, \omega^* \circ \sigma^*(v) = \ell \mid C_{p(i)}, \omega^* \circ \sigma^*(v) = \ell) \notag \\
    &\quad= \prod_{u \in R_+(v)} \E\bigg[ \Big(\frac{\overline{p}_{js}(x_{uv};\|u-v\|)}{\overline{p}_{j\ell}(x_{uv};\|u-v\|)}\Big)^t \: \Big| \: C_{p(i)}, \omega^* \circ \sigma^*(v) = \ell \bigg] \cdot \prod_{u \in R_-(v)} \E\bigg[ \Big(\frac{\overline{p}_{js}(x_{uv};\|u-v\|)}{\overline{p}_{j\ell}(x_{uv};\|u-v\|)}\Big)^t \: \Big| \: C_{p(i)}, \omega^* \circ \sigma^*(v) = \ell \bigg] \notag \\ 
    &\quad\leq \prod_{u \in R_+(v)} \phi_t(\overline{P}_{js}(\|u-v\|), \overline{P}_{j\ell}(\|u-v\|)) \cdot \prod_{u \in R_-(v)} \exp(\eta) \notag \\
    &\quad\leq e^{\eta M} \prod_{u \in R_+(v)} \phi_t(\overline{P}_{js}(\|u-v\|), \overline{P}_{j\ell}(\|u-v\|)).
\end{align}
In the first inequality, the first product holds because $\omega^*$ is permissible, meaning that conditioned on $\omega^* \circ \sigma^*(u) = j$ we have $x_{uv} \sim P_{\sigma^*(u)\sigma^*(v)} = P_{j\ell}$, and the second product holds due to Assumption \ref{assump:bounded_log_likelihood}. The second inequality holds because $\hat{\sigma}$ makes at most $M$ mistakes in $V_{p(i)}$.

Now, we utilize the properties of $C_{p(i)}$ which we conditioned on. Since vertex $v$ is $\Phi_2$-distinguished by $V_{p(i)}$, under Assumption \ref{assump:distinctness} there are at least $(\delta \log n)/(2k)$ vertices $u \in V_{p(i)}$ with $\hat{\sigma}(u) = j$ such that 
\begin{equation*}
    \phi_t(\overline{P}_{js}(\|u-v\|), \overline{P}_{j\ell}(\|u-v\|)) \leq \Phi_2
\end{equation*}
Then, since $\hat{\sigma}$ makes at most $M$ mistakes on $V_{p(i)}$, at least $(\delta \log n)/(2k) - M \geq (\delta \log n)/(3k)$ of those vertices are in $R_+(v)$. Therefore, we can upper-bound \eqref{eq:propagate_single_vertex_error_2} with
\begin{align*}
    \P(\hat{\sigma}(v) = s \mid C_{p(i)}, \omega^* \circ \sigma^*(v) = \ell) &\leq e^{\eta M} \Phi_2^{\frac{\delta \log n}{3k}} \\
    &= e^{\eta M} n^{-\delta \log(1 / \Phi_2) / (3k)} \\
    &= \eta_2 n^{-c_2},
\end{align*}
where the last equality follows from the definition of $\eta_2$ and $c_2$. Since this upper bound holds no matter what the conditioned true label $v$ is, the bound applies unconditionally on the true label of $v$ as well
\begin{equation*}
    \P(\hat{\sigma}(v) \neq \omega^* \circ \sigma^*(v) \mid C_{p(i)}) \leq \eta_2 n^{-c_2},
\end{equation*}
which upper-bounds the probability of making a mistake on a particular vertex $v \in V_i$.

Now, we upper-bound the probability of making more than $M$ mistakes on $V_i$. Let $K_i := |\{v \in V_i: \hat{\sigma}(v) \neq \omega^* \circ \sigma^*(v)\}|$ be the number of mistakes that $\hat{\sigma}$ makes on $V_i$, and let $\calE_v := \{\hat{\sigma}(v) \neq \omega^* \circ \sigma^*(v)\}$ be the event that a particular vertex $v$ is mislabeled by $\hat{\sigma}$. Conditioned on $C_{p(i)}$, the events $\calE_v$ for $v \in V_i$ are independent because disjoint sets of edges, which are independently generated conditioned on $C_{p(i)}$, are used to classify different vertices. Therefore, conditioned on $C_{p(i)}$ and $|V_i| < \Delta \log n$, we have that $K_i \preceq \text{Bin}(\Delta \log n, \eta_2 n^{-c_2})$. Letting $\mu := \eta_2 n^{-c_2} \Delta \log n$, the binomial Chernoff bound (Lemma \ref{lem:binomial_Chernoff}) yields
\begin{equation*}
\begin{aligned}
    \P(K_i > M \mid C_{p(i)}, |V_i| \leq \Delta \log n) &\leq \bigg(\frac{e^{(M/\mu) - 1}}{(M/\mu)^{(M/\mu)}}\bigg)^\mu \\
    &\leq e^M \Big(\frac{\mu}{M}\Big)^M \\
    &= e^M \Big(\frac{\eta_2 n^{-c_2} \Delta \log n}{M}\Big)^M \\
    &= e^M \Big(\frac{\eta_2 \Delta}{M}\Big)^M (\log n)^M n^{-c_2 M} \\
    &= \eta_3 (\log n)^M n^{-c_2 M}.
\end{aligned}
\end{equation*}
where the last equality follows from the definition of $\eta_3$. Finally, since $n^{-c_2 M} = n^{-5/4}$ by definition of $M$ and $(\log n)^M < n^{1/8}$ for sufficiently large $n$, we have that
\begin{equation*}
    \P(K_i > M \mid C_{p(i)}, |V_i| \leq \Delta \log n) \leq \eta_3 n^{-9/8} = o(1/n).
\end{equation*}
which shows that the probability of making more than $M$ mistakes on $V_i$ is $o(1/n)$.
\end{proof}

We now apply Proposition \ref{prop:propagation_error_block} to show that $\hat{\sigma}$ makes at most $M$ mistakes in all $\delta$-occupied blocks with probability $1 - o(1)$. Note that to apply Proposition \ref{prop:propagation_error_block}, all blocks must have at most $\Delta \log n$ vertices for some constant $\Delta$, which is given by a straightforward modification of Lemma D.1 in \cite{Gaudio+2024}.
\begin{lemma}
\label{lem:maximum_vertices_per_block}
    For the blocks obtained in Line \ref{line:construct_blocks} in Algorithm \ref{alg:exact_recovery}, there exists an explicitly computable constant $\Delta > 0$ such that 
    \begin{equation*}
        \P\bigg( \bigcap_{i = 1}^{n / (r^d \chi \log n)} \Big\{ |V_i| \leq \Delta \log n \Big\} \bigg) = 1 - o(1).
    \end{equation*}
\end{lemma}

We now prove that $\hat{\sigma}$ makes at most $M$ mistakes in all $\delta$-occupied blocks with high probability.

\begin{theorem}
\label{thm:almost_exact_recovery}
Let $G \sim \text{GHCM}(\lambda, n, r, \pi, P(y), d)$ such that $\lambda r > 1$ if $d = 1$ and $\lambda \nu_d r^d > 1$ if $d \geq 2$. Fix $\beta > 0$ and define $K := \nu_d(1 + \sqrt{d}\chi^{1/d})/\chi$. Suppose that $\chi, \delta > 0$ satisfy \eqref{eq:chi_delta_conditions_d=1} if $d = 1$ and \eqref{eq:chi_delta_conditions_d>=2} if $d \geq 2$, and also that $\delta < \beta / K$. Suppose that Assumptions \ref{assump:regularity_condition}, \ref{assump:bounded_log_likelihood}, and \ref{assump:distinctness} holds. Then, there exists a constant $M$ such that
\begin{equation}
\label{eq:errors_within_block}
    \P\bigg( \bigcap_{i \in V^\dagger} \Big\{ |v \in V_i: \hat{\sigma}(v) \neq \omega^* \circ \sigma^*(v) | \leq M \Big\} \bigg) = 1 - o(1),
\end{equation}
i.e. that $\hat{\sigma}$ makes at most $M$ mistakes on all $\delta$-occupied blocks with high probability. Consequently, $\hat{\sigma}$ achieves almost-exact recovery. Furthermore, it follows that
\begin{equation}
\label{eq:errors_within_neighborhood}
    \P\bigg( \bigcap_{v \in V} \{|u \in \calN(v): \hat{\sigma}(u) \neq \omega^* \circ \sigma^*(u)| \leq \beta \log n\} \bigg) = 1 - o(1),
\end{equation}
i.e. that $\hat{\sigma}$ makes at most $\eta \log n$ mistakes in the neighborhood of each vertex $v \in V$ with high probability.
\end{theorem}
\begin{proof}
    Let $\Delta > 0$ be such that Lemma \ref{lem:maximum_vertices_per_block} is satisfied and $0 < \Phi_2 < 1$ be such that Lemma \ref{lem:propagate_distinguishing_vertices} is satisfied. We define the events
    \begin{equation*}
    \begin{aligned}
        &\calI = \text{\{All blocks have at most $\Delta \log n$ vertices\},} \\
        &\calJ = \text{\{All vertices $v \in V \setminus V_{i_0}$ are $\Phi_2$-distinguished by their parent block\}, and} \\
        &\calH = \text{\{The $(r^d \chi \log n, \delta \log n)$-visibility graph of $G$ is connected\}}.
    \end{aligned}
    \end{equation*}
    Furthermore, let $\calA_i$ be the event that $\hat{\sigma}$ makes at most $M := 15k/(4\delta \log(1/\Phi_2))$ mistakes on $V_i$, i.e.
    \begin{equation*}
        \calA_i = \Big\{ |\{ v \in V_i: \hat{\sigma}(v) \neq \omega^* \circ  \sigma^*(v) \}| \leq M \Big\}.
    \end{equation*}
    Observe that
    \begin{align}
    \label{eq:Ai_all_blocks}
        \P\bigg( \bigcap_{i \in V^\dagger} \calA_i \bigg) &\geq \P\bigg( \bigcap_{i \in V^\dagger} \calA_i \mid \calI \cap \calJ \cap \calH \bigg) \P(\calI \cap \calJ \cap \calH) \notag \\
        &= \P(\calA_{i_0} \mid \calI \cap \calJ \cap \calH) \bigg( \prod_{j=1}^{|V^\dagger|} \P(\calA_{i_j} \mid \calA_{i_0}, \ldots, \calA_{i_{j-1}}, \calI \cap \calJ \cap \calH) \bigg) \P(\calI \cap \calJ \cap \calH).
    \end{align}
    We now lower-bound each factor in \eqref{eq:Ai_all_blocks}. For the first factor, we apply Proposition \ref{prop:MAP} to obtain that 
    \begin{equation}
    \label{eq:Ai_first_block}
        \P(\calA_{i_0} \mid \calI \cap \calJ \cap \calH) = 1 - o(1).
    \end{equation}
    For the second factor, we condition on the configuration $C_{p(i_j)}$ of $V_{p(i_j)}$ and use the law of total expectation. Define the event $\calB_j := \{\calA_{i_0} \cap \ldots \cap \calA_{i_{j-1}} \cap \calI \cap \calJ \cap \calH\}$. Then, we obtain that
    \begin{equation*}
    \begin{aligned}
        \P(\calA_{i_j} \mid \calB_j) &= \E[\1(\calA_{i_j}) \mid \calB_j] \\
        &= \E[ \E[\1(\calA_{i_j}) \mid \calB_j, C_{p(i_j)} ] \mid \calB_j ] \\
        &= \E[ \P(\calA_{i_j} \mid \calB_j, C_{p(i_j)}) \mid \calB_j].
    \end{aligned}
    \end{equation*}
    Now, observe that conditioned on the configuration $C_{p(i_j)}$, the event $\calA_{i_j}$ is independent of $\calB_j$ because the vertices in $V_{i_j}$ are labeled with respect to the vertices in $V_{p(i_j)}$, whose configuration is given. Therefore, we have that
    \begin{equation*}
        \P(\calA_{i_j} \mid \calB_j) = \E[ \P(\calA_{i_j} \mid C_{p(i_j)}) \mid \calB_j ].
    \end{equation*}
    Then, observe that conditioned on $\calB_j$, the configuration $C_{p(i_j)}$ satisfies 1) $\hat{\sigma}$ makes at most $M$ mistakes on $V_{p(i_j)}$ due to $\calA_{p(i_j)}$, 2) all vertices in $V_{i_j}$ are $\Phi_2$-distinguished by $V_{p(i_j)}$ due to $\calJ$, and 3) there are at most $\Delta \log n$ vertices in $V_{i_j}$ due to $\calI$. Thus, we can apply Proposition \ref{prop:propagation_error_block} to obtain that for any configuration $C_{p(i_j)}$ satisfying $\calB_j$, we have
    \begin{equation*}
        \P(\calA_{i_j} \mid C_{p(i_j)}) \geq 1 - \eta_3 n^{-9/8}
    \end{equation*}
    for a constant $\eta_3 > 0$. Therefore, we obtain that
    \begin{equation}
    \label{eq:Ai_remaining_blocks}
        \P(\calA_{i_j} \mid \calB_j) \geq 1 - \eta_3n^{-9/8}.
    \end{equation}
    For the third factor, we combine Proposition 1 from \cite{GaudioJin2025} establishing that the $\calH$ holds with high probability, Lemma \ref{lem:propagate_distinguishing_vertices}, and Lemma \ref{lem:maximum_vertices_per_block} to obtain
    \begin{equation}
    \label{eq:IJK_prob}
        \P(\calI \cap \calJ \cap \calH) = 1 - o(1).
    \end{equation}
    Finally, we can substitute \eqref{eq:Ai_first_block}, \eqref{eq:Ai_remaining_blocks}, and \eqref{eq:IJK_prob} into \eqref{eq:Ai_all_blocks}, and use Bernoulli's inequality to obtain
    \begin{equation*}
    \begin{aligned}
        \P\bigg( \bigcap_{i \in V^\dagger} \calA_i \bigg) &\geq (1 - o(1)) \left(1 - \eta_3 n^{-9/8}\right)^{n / (r^d \chi \log n)} (1 - o(1))\\
        &\geq (1 - o(1)) \left(1 - \frac{\eta_3 n^{-1/8}}{r^d \chi \log n}\right) (1 - o(1)) \\
        &= 1 - o(1),
    \end{aligned}
    \end{equation*}
    which shows (\ref{eq:errors_within_block}). Since $\delta$ can be arbitrarily small, we obtain that $\hat{\sigma}$ achieves almost-exact recovery.

    Now, we show \eqref{eq:errors_within_neighborhood}. From (\ref{eq:errors_within_block}), we use the fact that $\delta \log n > M$ for sufficiently large $n$ to obtain 
    \begin{equation*}
        \P\bigg( \bigcap_{i = 1}^{n/(r^d \chi \log n)} \Big\{ |v \in V_i: \hat{\sigma}(v) \neq \sigma^*(u_0) \sigma^*(v) | \leq \delta \log n \Big\} \bigg) = 1 - o(1)
    \end{equation*}
    since unoccupied blocks have fewer than $\delta \log n$ vertices, and thus mistakes. Then, we note that $K$ is an upper bound for the number of blocks which intersect the neighborhood $\calN(v)$ for any given vertex $v$. Therefore, we have that
    \begin{equation*}
        \P\bigg( \bigcap_{v \in V} \{|u \in \calN(v): \hat{\sigma}(u) \neq \sigma^*(u_0)\sigma^*(u)| \leq \delta K \log n\} \bigg) = 1 - o(1).
    \end{equation*}
    Since $\delta < \beta / K$, we obtain (\ref{eq:errors_within_neighborhood}).
\end{proof}

\subsection{Refine}
\label{subsec:refine}
In this section, we will show that Algorithm \ref{alg:exact_recovery} achieves exact recovery under the assumptions of Theorem \ref{thm:achievability}, adapting the proof of \cite{Gaudio+2025}. To accomplish this, we will show the probability that \texttt{Refine} makes a mistake on any given vertex $v\in V$, is $o(1/n)$. Then, taking a union bound over all vertices will show that \texttt{Refine} labels all vertices correctly with high probability. We will use the following formula for the cumulant-generating function (CGF) of a random sum of random variables.

\begin{fact}
\label{fact:CGF_formula}
    Let $Y$ be a random variable, and $X_1, X_2, \ldots, X_Y$ be $Y$ i.i.d. copies of a random variable $X$ independent of $Y$. Let $S = \sum_{i=1}^Y X_i$. Then, $\Lambda_S(t) = \Lambda_Y(\Lambda_X(t))$,
\end{fact}

We first upper-bound the probability that the MAP estimator comes ``close'' to making an error on any given vertex $v \in V$ correctly, where the posterior likelihood is computed with respect to the true labeling $\sigma^*$.

\begin{lemma}
\label{lem:correct_labels_MAP_failures}
    If $\lambda \nu_d r^d \min_{i \neq j} D_+(\theta_i \Vert \theta_j; \pi, g) > 1$, then for a fixed $0 < \epsilon \leq (\lambda \nu_d r^d \min_{i \neq j} D_+(\theta_i \Vert \theta_j; \pi, g) - 1 ) / 2$, we have that
    \begin{equation*}
        \P(\ell_i(v, \sigma^*) - \ell_j(v, \sigma^*) \leq \epsilon \log n \mid \sigma^*(v) = i) = n^{-(1 + \Omega(1))}
    \end{equation*}
    where
    \begin{equation*}
        \ell_i(v, \sigma) = \sum_{v \in V: u \visible v} \log(\overline{p}_{i, \sigma(u)}(x_{uv};\|u - v\|)).
    \end{equation*}
\end{lemma}
\begin{proof}
    Fix a vertex $v \in V$ with $\sigma^*(v) = i$. We construct a random variable $X$ such that 
    \begin{equation*}
        X \sim \ell_j(v, \sigma^*) - \ell_i(v, \sigma^*).
    \end{equation*}
    Observe that
    \begin{align*}
         \ell_j(v, \sigma^*) - \ell_i(v, \sigma^*) &= \sum_{u \in V: u \visible v} \log\bigg(\frac{p_{j, \sigma^*(v)}(x_{uv}; \|u-v\|)}{p_{i, \sigma^*(v)}(x_{uv}; \|u-v\|)}\bigg) \\
         &= \sum_{s \in Z} \bigg( \sum_{u \in V: u \visible v, \sigma^*(u) = s} \log\Big(\frac{p_{js}(x_{uv}; \|u-v\|)}{p_{is}(x_{uv}; \|u-v\|)}\Big) \bigg)
    \end{align*}
    Now, we construct $X_s$ for each $s \in Z$ such that
    \begin{equation*}
        X_s \sim \log\bigg(\frac{p_{js}(x_{uv}; \|u-v\|)}{p_{is}(x_{uv}; \|u-v\|)}\bigg).
    \end{equation*}
    Let $D$ be a random variable with density
    \begin{equation*}
        f_D(x) :=
        \begin{cases}
            dx^{d-1}/r^d \text{ if } 0 \leq x \leq r \\
            0 \text{ otherwise,} 
        \end{cases}
    \end{equation*}
    which represents the distance between $v$ and a vertex $u$ placed uniformly at random in the neighborhood of $v$. Let $X_{is}$ be a random variable coupled to $D$, such that conditioned on $D = y$, we sample $X_{is}$ from $P_{is}(y)$. Note that $X_{is}$ represents the edge weight between $u$ and $v$. Then, we define
    \begin{equation*}
        X_s := \log\bigg( \frac{p_{js}(X_{is} ;D)}{p_{is}(X_{is} ;D)} \bigg).
    \end{equation*}  
    Now, letting $N_s \sim \text{Pois}(\lambda \nu_d r^d \pi_s \log n)$ and $\{X_s^{(m)}\}_{m \in \N} \overset{\text{iid}}{\sim} X_s$, we can define
    \begin{equation*}
        X = \sum_{s \in Z}\bigg( \sum_{m=1}^{N_s} X_s \bigg).
    \end{equation*}
    By the Chernoff bound, we have that for any fixed $\epsilon > 0$,
    \begin{align}
        \P(\ell_i(v, \sigma^*) - \ell_j(v, \sigma^*) \leq \epsilon \log n \mid \sigma^*(v) = i) &= \P(X \geq -\epsilon \log n) \nonumber \\
        &\leq \inf_{t \in [0, 1]} \E[e^{tX}] e^{t\epsilon \log n}  \nonumber \\
        &= \inf_{t \in [0, 1]} \E[e^{tX}]n^{t\epsilon}.     \label{eq:Chernoff_bound_X}
    \end{align}
    Thus, we need to compute the moment-generating function (MGF) of $X$. Observe that
    \begin{align}
    \label{eq:MGF_X}
        \E[e^{tX}] &= \E\bigg[\exp\bigg(t\sum_{s \in Z} \sum_{m=1}^{N_s} X_s\bigg)\bigg] \notag \\
        &= \prod_{s \in Z} \E\bigg[\exp\bigg(t\sum_{m=1}^{N_s} X_s\bigg)\bigg] \notag \\
        &= \prod_{s \in Z} \exp\Big(\Lambda_{N_s}(\Lambda_{X_s}(t))\Big).
    \end{align}
    Computing the MGF of $X_s$ using the law of total expectation, we obtain that
    \begin{align*}
        \E[e^{tX_s}] &= \E\bigg[\exp\bigg(t \log\Big( \frac{p_{js}(X_{is};D)}{p_{is}(X_{is};D)} \Big)\bigg)\bigg] \\
        &= \E\bigg[\E\bigg[\exp\bigg(t \log\Big( \frac{p_{js}(X_{is};D)}{p_{is}(X_{is};D)} \Big)\bigg) \mid D = y\bigg]\bigg] \\
        &= \int_0^r \E\bigg[ \exp\bigg(t \log\Big( \frac{p_{js}(X_{is};y)}{p_{is}(X_{is};y)} \Big)\bigg)\bigg] \frac{dy^{d-1}}{r^d} dy \\
        &= \int_0^r \phi_t(P_{js}(y), P_{is}(y)) \frac{dy^{d-1}}{r^d} dy \\
        &= \overline{\phi}_t(P_{js}, P_{is}).
    \end{align*}
    Hence, $\Lambda_{X_s}(t) = \log \overline{\phi}_t(P_{js}, P_{is})$. Now, combining the fact that $\Lambda_{N_s}(t) = (\lambda \nu_d r^d \pi_s \log n)(e^t - 1)$ with Fact \ref{fact:CGF_formula}, we obtain that
    \begin{align}
    \label{eq:CGF_calculation}
        \Lambda_{N_s}(\Lambda_{X_s}(t)) &= (\lambda \nu_d r^d \pi_s \log n) (\overline{\phi}_t(P_{js}, P_{is}) - 1) \notag \\
        &= - \lambda \nu_d r^d \pi_s (1 - \overline{\phi}_t(P_{js}, P_{is})) \log n.
    \end{align}
    Substituting \eqref{eq:CGF_calculation} into \eqref{eq:MGF_X} yields the MGF of $X$, which is
    \begin{align}
    \label{eq:MGF_X_finished}
        \E[e^{tX}] &= \prod_{s \in Z} \exp\bigg( - \lambda \nu_d r^d \pi_s (1 - \overline{\phi}_t(P_{js}, P_{is})) \log n \bigg) \notag \\
        &= \prod_{s \in Z} n^{- \lambda \nu_d r^d \pi_s (1 - \overline{\phi}_t(P_{js}, P_{is}))} \notag \\
        &= n^{-\lambda \nu_d r^d \big( 1 - \sum_{s \in Z} \pi_s \overline{\phi}_t(P_{js}, P_{is}) \big)}.
    \end{align}
    Finally, substituting \eqref{eq:MGF_X_finished} into \eqref{eq:Chernoff_bound_X}, we obtain that for any fixed $\epsilon > 0$,
    \begin{align*}
        \P(\ell_i(v, \sigma^*) - \ell_j(v, \sigma^*) \leq \epsilon \log n \mid \sigma^*(v) = i) &\leq \inf_{t \in [0, 1]} n^{-\lambda \nu_d r^d \big( 1 - \sum_{s \in Z} \pi_s \overline{\phi}_t(P_{js}, P_{is}) \big) + t\epsilon} \\
        &\leq n^{-\lambda \nu_d r^d \big( 1 - \inf_{t \in [0, 1]} \sum_{s \in Z} \pi_s \overline{\phi}_t(P_{js}, P_{is}) \big) + \epsilon} \\
        &= n^{-\big(\lambda \nu_d r^d D_+(\theta_i \Vert \theta_j; \pi, g) - \epsilon \big)}.
    \end{align*}
    Since $\lambda \nu_d r^d \min_{i \neq j} D_+(\theta_i \Vert \theta_j; \pi, g) > 1$ by the achievability criterion, we can take $\epsilon = (\lambda \nu_d r^d \min_{i \neq j} D_+(\theta_i \Vert \theta_j; \pi, g) - 1)/2$ to obtain that
    \begin{equation*}
        \P(\ell_i(v, \sigma^*) - \ell_j(v, \sigma^*) \leq \epsilon \log n \mid \sigma^*(v) = i) = n^{-(1 + \Omega(1))}.
    \end{equation*}
\end{proof}

Lemma \ref{lem:correct_labels_MAP_failures} upper-bounds the probability that the MAP estimator with respect to $\sigma^*$ comes ``close'' to making an error. To finish proving Theorem \ref{thm:achievability}, we upper-bound the error probability of the MAP estimator with respect to $\hat{\sigma}$ by showing that the posterior likelihoods with respect to $\sigma^*$ and $\hat{\sigma}$ are close.

\begin{proof}[Proof of Theorem \ref{thm:achievability}] The remainder of the proof follows exactly from Theorem E.3 in \cite{Gaudio+2025}. For completeness, we reproduce the main steps here. Recall that $\omega^{\star}$ is such that $\hat{\sigma}_0 = \omega^{\star} \circ \sigma^{\star}$, where $\hat{\sigma}_0$ is the initial block labeling.
Fix $\beta > 0$ and let the event $\calE_1$ be the event that $\hat{\sigma}$ makes at most $\beta \log n$ mistakes in the neighborhood of every vertex $v \in V$. That is,
\begin{equation*}
    \mathcal{E}_1 = \bigcap_{v \in V}\Big\{ |\{u \in \calN(v): \hat{\sigma}(u) \neq \omega^* \circ \sigma^*(u) \}| \leq \beta \log n \Big\}.
\end{equation*}
By Theorem \ref{thm:almost_exact_recovery}, $\P(\calE_1) = 1 - o(1)$.

Rather than analyze $\hat{\sigma}$ directly, we will provide a uniform upper bound on the failure probability of the MAP estimator on a given vertex $v \in V$, over all labelings $\sigma$ which are ``close'' to the true labels $\sigma^*$. Thus, fix $v \in V$ and let $W(v, \beta)$ be the set of labelings $\sigma$ which differ from $\omega^* \circ \sigma^*$ by at most $\beta \log n$ vertices in the neighborhood of $v$, meaning that
\begin{equation*}
    W(v; \beta) := \{\sigma: |\{u \in \calN(v): \sigma(u) \neq \omega^* \circ \sigma^*(u)\}| \leq \beta \log n\}.
\end{equation*}

Now, we define the event
\begin{equation*}
    \calE_v = \bigcup_{i \in Z} \bigg[ \{\omega^* \circ \sigma^*(v) = i\} \bigcap \bigg( \bigcup_{\sigma \in W(v; \beta)} \bigcup_{j \neq i} \{\ell_i(v, \sigma) \leq \ell_j(v, \sigma) \}\bigg)\bigg],
\end{equation*}
which is the event that there exists a labeling $\sigma \in W(v; \beta)$ such that the MAP estimator with respect to $\sigma$ fails on $v$. Due to the fact that $\mathcal{E}_1$ holds with high probability, and there are $\Theta(n)$ vertices in total with high probability, it intuitively suffices to show that $\mathcal{E}(v)$ holds with probability $o(1/n)$. However, we cannot simply take a union bound over the vertex set, as the number of vertices is random. The details for nevertheless constructing a union bound can be found in \cite{Gaudio+2025}.

In order to bound $\mathbb{P}(\mathcal{E}_v)$, we leverage Assumption \ref{assump:bounded_log_likelihood} to obtain that 
\begin{equation*}
    |\ell_i(v, \omega^* \circ \sigma^*) - \ell(v, \sigma)| \leq \eta \beta \log n.
\end{equation*}
for all $\sigma \in W(v; \beta)$. Thus, the event
$\bigcup_{\sigma \in W(v; \beta)}  \{\ell_i(v, \sigma) \leq \ell_j(v, \sigma) \}$ occurring within $\mathcal{E}_v$ implies 
\begin{equation}
\{\ell_i(v,\sigma^{\star}) \leq \ell_j(v, \sigma^{\star}) + 2 \eta \beta \log n\}. \label{eq:refine-implication}
\end{equation} 
Note that $\beta > 0$ can be arbitrarily small.
Conditioned on $\sigma^{\star}(v) = i$, Lemma \ref{lem:correct_labels_MAP_failures} implies that \eqref{eq:refine-implication} occurs with probability $n^{-(1 + \Omega(1))}$ when $\beta > 0$ is small enough, thus establishing that $\mathcal{E}_v$ holds with probability $n^{-(1 + \Omega(1))}$ also.
\end{proof}

\section{Exact Recovery for $d = 1, |\Omega_{\pi, P}| = 1$}
\label{appx:1d_proof_of_exact_recovery}
In this section, we present Algorithm \ref{alg:exact_recovery_1d} and show that it achieves exact recovery when $d = 1, |\Omega_{\pi, P}| = 1$ under Assumptions \ref{assump:regularity_condition}, \ref{assump:bounded_log_likelihood}, and \ref{assump:strong_distinctness}. We note that this algorithm is a minor adaption of Algorithm 6 in \cite{Gaudio+2025}. The algorithm starts by partitioning $\calS_{1, n}$ into blocks of length $(r \log n) / 2$, so that any two adjacent blocks are mutually visible. Then, for a sufficient small constant $\delta > 0$, we form a block visibility graph $H = (V^\dagger, E^\dagger)$ by letting $V^\dagger$ be the set of all $\delta$-occupied blocks (i.e. blocks with at least $\delta \log n$ vertices) and $E^\dagger$ be all pairs of adjacent, $\delta$-occupied blocks. Since $\lambda$ is arbitrary, the block visibility graph need not be connected; therefore, we must label each connected component of the visibility graph independently from each other. For each connected component (which we also call a \textit{segment}), we first label an initial block using the \texttt{Maximum a Posteriori}, then label the remaining blocks in the segment block-by-block via \texttt{Propagate}.

We note that by a minor modification of Lemma D.10 in \cite{Gaudio+2025}, there exists $\delta > 0$ such that there are $o(n^{1 - \lambda r/4})$ segments.
\begin{lemma}
\label{lem:max_initial_blocks_1d}
    There exists $\delta > 0$ such that the number of $\delta$-unoccupied blocks at most $n^{1-\lambda/4}$ with high probability. Hence, the number of segments is at most $n^{1-\lambda/4}$ with high probability.
\end{lemma}

\begin{algorithm}
\caption{Exact Recovery for $d = 1, |\Omega_{\pi, P}| = 1$}
\label{alg:exact_recovery_1d}
\begin{algorithmic}[1]
    \Require $G \sim \text{GHCM}(\lambda, n, r, \pi, P(y), 1)$.
    \Ensure An estimated community labeling $\tilde{\sigma}: V \to Z$.
    \State \textbf{Phase I:}
    \State Divide $\calS_{1, n}$ into $2n/(r \log n)$ blocks of length $(r \log n) /2$ each. \label{line:create_blocks_1d}
    \State Let $\delta > 0$ be such that Lemma \ref{lem:max_initial_blocks_1d} is satisfied. 
    \State Construct the $(r \log n)/2, \delta \log n)$ block visibility graph $H = (V^\dagger, E^\dagger)$. Let $\ell$ be the number of segments and $n_m$ be the number of blocks in each segment for $m \in [\ell]$. Let $B_i^m$ be the $i^\text{th}$ block of the $m^\text{th}$ segment and $V_i^m$ be the set of vertices in  $B_i^m$ for $i \in [n_m]$.
    \State Set $\epsilon_0 \leq \min\{\frac{\lambda r}{4 \log k}, \delta\}$.
    \For{$m = 1, 2, \ldots, \ell$}
        \State Sample $V_0^m \subset V_1^m$ such that $|V_0^m| = \epsilon_0 \log n$. Set $V_1^m \gets V_1^m \setminus V_0^m$.
        \State Apply \texttt{Maximum a Posteriori} (Algorithm \ref{alg:MAP}) on input $(G, V_0^r)$ to obtain the labeling $\hat{\sigma}$ on $V_0^m$.
        \For{$i = 1, 2, \ldots, n_m$}
            \State Apply \texttt{Propagate} (Algorithm \ref{alg:propagate}) on input $(G, V_{i-1}^m, V_i^m)$ to obtain the labeling $\hat{\sigma}$ on $V_1^m$.
        \EndFor
    \EndFor
    \State \textbf{Phase II:}
    \For{$v \in V$} 
        \State Apply \texttt{Refine} (Algorithm \ref{alg:Refine}) on input $(G, v, \hat{\sigma})$ to compute $\tilde{\sigma}(v)$.
    \EndFor
\end{algorithmic}
\end{algorithm}

\subsection{Maximum a Posteriori}
Since the block visibility graph can be disconnected, we use \texttt{Maximum a Posteriori} to label an initial block in each segment. Like in \cite{Gaudio+2025}, we show that for any $c > 0$, we make at most $c \log n$ mistakes on all initial blocks $V_0^m$ for $m = 1, 2, \ldots, \ell$ with high probability. 
\begin{prop}
\label{prop:MAP_1d}
        Let $d = 1, |\Omega_{\pi, P}| = 1$. Let $c > 0$. Suppose that Assumption \ref{assump:strong_distinctness} is satisfied. Then, the estimator $\hat{\sigma}$ produced by Phase I of Algorithm \ref{alg:exact_recovery_1d} makes at most $c \log n$ mistakes on all initial blocks with high probability.
    \end{prop}
\begin{proof}
    We first consider the probability of making more than $c \log n$ mistakes for a given initial block $V_0^m$. Fix $\sigma$ such that $\text{DISC}(\sigma, \sigma^*) > c \log n$ on an initial block $V_0^m$. As in the proof of Proposition \ref{prop:high_discrepancy}, we have that conditioned on the locations of vertices given by $L_0^m$
    \begin{align}
    \label{eq:MAP_1d_error_eq1}
        \P(\ell_0(G, \sigma) - \ell_0(G, \sigma^*) > 0 \mid L_0^m)  &\leq \E[\exp\big(t(\ell_0(G, \sigma) - \ell_0(G, \sigma^*))\big) \mid L_0^m] \nonumber \\
        &= \prod_{u \in V_0^m} \prod_{v \in V_0^m, v \neq u} \phi_t(\overline{P}_{\sigma(u)\sigma(v)}(\|u-v\|), \overline{P}_{\sigma^*(u)\sigma^*(v)}(\|u-v\|)).
    \end{align}
    Let $n_{ij} = \{v \in V_0^m: \sigma^*(v) = i, \sigma(v) = j\}$. 
    By the pigeonhole principle, there exists $i \neq j$ such that $n_{ij} \geq \frac{c \log n}{k(k-1)}$. Then, defining $\epsilon > 0$ as in \eqref{eq:def_c_epsilon} and using the same reasoning as in \cite[Lemma D.11]{Gaudio+2025}, there exist $a, b \in Z$ such that $n_{ab} \geq (\epsilon \epsilon_0 \log n) / k$ with high probability. Then, we note that \eqref{eq:MAP_1d_error_eq1} is upper-bounded by
    \begin{equation}
    \label{eq:MAP_1d_error_eq2}
        \P(\ell_0(G, \sigma) - \ell_0(G, \sigma^*) > 0 \mid L_0^m ) \leq \prod_{\substack{u \in V_0^m \\ \sigma^*(u) = i, \sigma(u) = j}} \prod_{\substack{v \in V_0^m , v \neq u \\ \sigma^*(v) = a, \sigma(v) = b}} \phi_t(\overline{P}_{jb}(\|u-v\|), \overline{P}_{ia}(\|u-v\|))
    \end{equation}
    Fix a vertex $u \in V_0^m$, and let $D$ be the distance between $u$ and a vertex placed uniformly at random in the initial block $B_1^m$. Fix $\alpha > 0$, and let $\Phi_1 < 1$ be such that $\phi_t(\overline{P}_{i'j'}(D), \overline{P}_{a'b'}(D)) \geq \Phi_1$ for all $i',j',a',b' \in Z$ with $P_{i'j'} \not \equiv P_{a'b'}$ with probability at most $\alpha$ via Corollary 1. Using a calculation analogous to that in Lemma \ref{lem:propagate_distinguishing_vertices}, we obtain that for sufficiently small $\alpha$, the number of vertices $v \in V_0^m: \sigma^*(v) = a, \sigma(v) = b$ such that 
    \begin{equation*}
        \phi_t(\overline{P}_{i'j'}(\|u-v\|), \overline{P}_{a'b'}(\|u-v\|)) < \Phi_1
    \end{equation*}
    for all $i',j',a',b' \in Z$ with $P_{i'j'} \not \equiv P_{a'b'}$ is at least $(\epsilon \epsilon_0 \log n) / (2k)$ with probability $1 - n^{-(1 + \Omega(1))}$. Since there are at most $\epsilon_0 \log n$ vertices per initial block and $n^{1 - \lambda/4}$ initial blocks, taking a union bound shows that this holds simultaneously for all $u \in V_0^m$ and $m \in [\ell]$ (i.e. for all vertices in the set of initial blocks) with probability $1 - o(1)$. For a particular block $B_0^m$, we denote the event that all $u \in V_0^m$ have at least $(\epsilon \epsilon_0 \log n)/(2k)$ vertices $v \in V_0^m$ such that
    \begin{equation*}
        \phi_t(\overline{P}_{ij}(\|u-v\|), \overline{P}_{i'j'}(\|u-v\|)) < \Phi_1 
    \end{equation*}
    for all $i',j',a',b' \in Z$ with $P_{i'j'} \not \equiv P_{a'b'}$ as $\{L_0^m \in \calL_\text{good}^m$\}.
    
    Now, due to Assumption \ref{assump:strong_distinctness}, we have that $P_{jb} \not \equiv P_{ia}$. Thus, the number of vertices $v \in V_0^m$ with $\sigma^*(v) = a, \sigma(v) = b$ for which 
    \begin{equation*}
        \phi_t(\overline{P}_{ij}(\|u-v\|), \overline{P}_{ab}(\|u-v\|)) < \Phi_1
    \end{equation*}
    is at least $(\epsilon \epsilon_0 \log n)/(2k)$ whenever $L_0^m \in \calL_\text{good}^m$. Hence, we can upper-bound \eqref{eq:MAP_1d_error_eq2} conditioned on $\{L_0^m \in \calL_\text{good}^m\}$, which yields
    \begin{equation*}
        \P(\ell_0(G, \sigma) - \ell_0(G, \sigma^*) > 0 \mid L_0^m \in \calL_\text{good}^m ) \leq \Phi_1^{\frac{c \log n}{k(k-1)} \frac{\epsilon \epsilon_0 \log n}{2k}} = C^{\log^2 n} = n^{-\Omega(\log n)}
    \end{equation*}
    where $C < 1$ is a constant. Taking a union bound over the $k^{\epsilon_0 \log n}$ possible labelings, we obtain
    \begin{equation*}
        \P(\bigcup_{\sigma: \text{DISC}(\sigma, \sigma^*) > c \log n}\ell_0(G, \sigma) - \ell_0(G, \sigma^*) > 0 \mid L_0^m \in \calL_\text{good}^m ) \leq k^{\epsilon_0 \log n} n^{-\Omega(\log n)} = n^{-\Omega(\log n)},
    \end{equation*}
    which implies that
    \begin{equation*}
        \P( |\{u \in V_0^m: \hat{\sigma}(u) \neq \omega^* \circ \sigma^*(u)\}| > c \log n \mid L_0^m \in \calL_\text{good}^m) = n^{-\Omega(\log n)}.
    \end{equation*}
    Hence, we have an upper bound for the probability that conditioned on $L_0^m \in \calL_\text{good}^m$, the MAP estimator $\hat{\sigma}$ makes more than  $c \log n$ mistakes on $V_0^m$. 
    
    We now upper-bound the probability that the MAP estimator makes more than $c \log n$ mistakes in any initial block via the union bound. Let $L_0$ denote the locations of vertices in all initial blocks, and define the event $\{L_0 \in \calL_\text{good}\} := \bigcap_{m \in [\ell]} \{L_0^m \in \calL_\text{good}^m\}$. Taking a union bound over $n^{1 - \lambda/4}$ initial blocks, we obtain that conditioned on $\{L_0 \in \calL_\text{good}\}$, the MAP estimator $\hat{\sigma}$ makes more than $c \log n$ mistakes on any initial block with probability at most $n^{1-\lambda r/4} n^{-\Omega(\log n)} = o(1)$. That is,
    \begin{equation*}
        \P\bigg( \bigcup_{m \in [\ell]} \Big\{ |\{u \in V_0^m: \hat{\sigma}(u) \neq \omega^* \circ \sigma^*(u)\}| > c \log n \Big\} \mid L_0 \in \calL_\text{good}\bigg) = o(1).
    \end{equation*}
    Finally, since $\P(L_0 \in \calL_\text{good}) = 1 - o(1)$ by definition of $\Phi_1$, we obtain that the unconditional probability is $o(1)$ as well, i.e.
    \begin{equation*}
        \P\bigg( \bigcup_{m \in [\ell]} \Big\{ |\{u \in V_0^m: \hat{\sigma}(u) \neq \omega^* \circ \sigma^*(u)\}| > c \log n \Big\} \bigg) = o(1).
    \end{equation*}
\end{proof}

\subsection{Propagate}
For each segment, we label the remaining blocks using \texttt{Propagate}. We will show that for any $c > 0$, we make less than $c \log n$ mistakes on all blocks labeled with Propagate using the strategy of \cite{Gaudio+2025}. Let $\calA_i^m$ be the event that $\hat{\sigma}$ makes at most $c \log n$ mistakes in $V_i^m$ and let $C_i^m$ denote the configuration of $V_i^m$.

The following lemma upper-bounds the probability of making a mistake on any given vertex $v \in V_i^m$.
\begin{lemma}
    Let $d = 1, |\Omega_{\pi, P}| = 1$. Suppose that Assumptions \ref{assump:bounded_log_likelihood} and \ref{assump:strong_distinctness} are satisfied. Fix $\delta > 0$ and let $\Phi_2$ be such that Lemma \ref{lem:propagate_distinguishing_vertices} is satisfied.    Let $c > 0$ be such that 
    \begin{equation}
    \label{eq:choice_c_propagate_1d}
        c < \frac{\delta \log(1 / \Phi_2)}{2k(\eta+\log(1/\Phi_2))}.
    \end{equation}  
    Let $B_i^m$ be a $\delta$-occupied block with vertices $V_i^m$, and let $B_{p(i)}^m$ with vertices $V_{p(i)}^m$ be its parent block. Suppose that $V_{p(i)}^m$ has configuration $C_{p(i)}^m$ such that $\hat{\sigma}$ makes at most $c \log n$ mistakes on $V_{p(i)}^m$ and that all vertices $v \in V_i$ are $\Phi_2$-distinguished by $V_{p(i)}^m$. Suppose that $\hat{\sigma}$ labels $V_i^m$ using Propagate. Then, there exists a constant $c_3 > 0$ such that for any given $v \in V_i^m$, the probability of $\hat{\sigma}$ labeling $v$ incorrectly is
    \begin{equation}
    \label{eq:propagate_1d_single_vertex_bound}
        \P(\hat{\sigma}(v) \neq \sigma^*(v) \mid C_{p(i)}^m) \leq n^{-c_3}
    \end{equation}
    Furthermore, suppose that $|V_i^m| < \Delta \log n$ for some constant $\Delta > 0$. Then, there exist constants $c_4$ and $c_5$ such that the probability of $\hat{\sigma}$ making more than $c \log n$ mistakes on $V_i^m$ is
    \begin{equation}
    \label{eq:propagate_1d_block_bound}
        \P\bigg( | \{v \in V_i^m: \hat{\sigma}(v) \neq \sigma^*(v)\} | > c \log n \: \Big| \: C_{p(i)}^m, |V_i^m| \leq \Delta \log n \bigg) \leq c_4^{\log n} n^{-c_5 \log n}.
    \end{equation}
\end{lemma}
\begin{proof}
    Let $\sigma^*(v) = a$. We first upper bound $\P(\hat{\sigma}(v) = s \mid C_{p(i)}^m, \sigma^*(v) = a)$, i.e. the probability that $\hat{\sigma}$ labels $v$ incorrectly as $s \neq a$. Following the reasoning of Proposition \ref{prop:propagation_error_block}, under Assumptions \ref{assump:bounded_log_likelihood} and \ref{assump:strong_distinctness} we obtain that
    \begin{align*}
        \P(\hat{\sigma}(v) = s \mid C_{p(i)}^m, \sigma^*(v) = a) &\leq e^{\eta c \log n} \Phi_2^{(\delta \log n)/(2k) - c \log n} \\
        &= \Phi_2^{(\delta \log n)/(2k)} \left(\frac{e^\eta}{\Phi_2}\right)^{c \log n} \\
        &= n^{-\delta \log(1 / \Phi_2)/(2k)} e^{(\eta + \log(1/\Phi_2))c \log n} \\
        &= n^{-\delta \log(1 / \Phi_2)/(2k)} n^{(\eta + \log(1/\Phi_2))c} \\
        &= n^{-c_3}
    \end{align*}
    where $c_3 := -\delta \log(1 / \Phi_2)/(2k) + c(\eta+\log(1/\Phi_2)) > 0$ by the choice of $c$ in \eqref{eq:choice_c_propagate_1d}. Since this holds for all communities $s \neq a$, we obtain
    \begin{equation*}
        \P(\hat{\sigma}(v) \neq \sigma^*(v) \mid C_{p(i)}^m) \leq n^{-c_3},
    \end{equation*}
    thus showing \eqref{eq:propagate_1d_single_vertex_bound}.

    We now upper bound the probability that $\hat{\sigma}$ makes more than $c \log n$ mistakes on $V_i^m$. Let $K_i^m := |\{v \in V_i^m: \hat{\sigma}(v) \neq \sigma^*(v)\}|$ be the number of mistakes that $\hat{\sigma}$ makes on $V_i^m$, and let $\calE_v := \{\hat{\sigma}(v) \neq \sigma^*(v)\}$ be the event that a particular vertex $v$ is mislabeled by $\hat{\sigma}$. Conditioned on $C_{p(i)}^m$ and $|V_i^m| < \Delta \log n$, we have that $K_i^m \preceq \text{Bin}(\Delta \log n, n^{-c_3})$. Letting $\mu := n^{-c_3} \Delta \log n$, the binomial Chernoff bound (Lemma \ref{lem:binomial_Chernoff}) yields
    \begin{align*}
        \P(\text{Bin}(\Delta \log n, n^{-c_3}) > c \log n) &\leq \left(\frac{e^{(c \log n - \mu) / \mu}}{(c \log(n) / \mu)^{c \log (n) / \mu}}\right) \\
        &= e^{c \log n - \mu} \left(\frac{\mu}{c \log n}\right)^{c \log n} \\
        &\leq \left(\frac{e \mu}{c \log n}\right)^{c \log n} \\
        &= \left(\frac{e n^{-c_3} \Delta \log n}{c \log n}\right)^{c \log n} \\
        &= \left(\frac{e \Delta}{c}\right)^{c \log n} \cdot n^{-c_3 c \log n} \\
        &= c_4^{\log n} n^{-c_5 \log n}
    \end{align*}
    where $c_4 := (e\Delta/c)^c$ and $c_5 := c_3 c$, thus showing \eqref{eq:propagate_1d_block_bound}.
\end{proof}

We now prove that $\hat{\sigma}$ makes at most $c \log n$ mistakes on all $\delta$-occupied blocks with high probability. Like before, a minor modification to Lemma D.1 in \cite{Gaudio+2025} shows that there exists a constant $\Delta > 0$ such that all blocks have at most $\Delta \log n$ vertices with high probability.

\begin{lemma}
\label{lem:maximum_vertices_per_block_1d}
For the blocks obtained in Line \ref{line:create_blocks_1d} of Algorithm \ref{alg:exact_recovery_1d}, there exists $\Delta > 0$ such that
\begin{equation*}
    \P(\bigcap_{i=1}^{2n/(r \log n)} \{|V_i| \leq \Delta \log n\}) = 1 - o(1).
\end{equation*}
\end{lemma}

The proof to establish almost-exact recovery is very similar to the proof of Theorem \ref{thm:almost_exact_recovery}.

\begin{theorem}
Let $d = 1, |\Omega_{\pi, P}| = 1$. Suppose that $\delta > 0$ satisfies Lemma \ref{lem:max_initial_blocks_1d} and $\Phi_2$ satisfies Lemma \ref{lem:propagate_distinguishing_vertices}. Suppose that Assumptions \ref{assump:bounded_log_likelihood} and \ref{assump:strong_distinctness} are satisfied. Then, for any $c > 0$ such that \eqref{eq:choice_c_propagate_1d} holds, the Phase I estimator $\hat{\sigma}$ makes at most $c \log n$ mistakes on all $\delta$-occupied blocks with high probability. Consequently, $\hat{\sigma}$ achieves almost-exact recovery.
\end{theorem}
\begin{proof}
Let $\Delta > 0$ be such that Lemma \ref{lem:maximum_vertices_per_block_1d} is satisfied, and $0 < \Phi_2 < 1$ such that Lemma \ref{lem:propagate_distinguishing_vertices} is satisfied. Define the events 
\begin{align*}
    \calI &= \text{\{All blocks have at most $\Delta \log n$ vertices\}}, \\
    \calJ &= \text{\{All vertices $v \in V \setminus \cup_{m=1}^\ell V_0^m$ are $\Phi_2$-distinguished by their parent block\}, and} \\
    \calH &= \text{\{There are most $n^{1-\lambda r /4}$ segments\}}.
\end{align*}
Then, let $\calA_i^m$ be the event that $\hat{\sigma}$ makes at most $c \log n$ mistakes $V_{i}^m$. Observe that
\begin{align}
\label{eq:Ai_all_blocks_1d}
    \P\bigg(\bigcap_{m=1}^\ell \bigcap_{i=1}^{n_m} \calA_i^m\bigg) &\geq \P\bigg(\bigcap_{m=1}^\ell \bigcap_{i=1}^{n_m} \calA_i^m \mid \calI \cap \calJ \cap \calH\bigg) \P(\calI \cap \calJ \cap \calH) \notag \\
    &= \bigg[\prod_{m=1}^\ell \P(\calA_0^m \mid \calI \cap \calJ \cap \calH) \prod_{i=1}^{n_m} \P(\calA_i^m \mid \calA_0^m, \ldots, \calA_{i-1}^m, \calI \cap \calJ \cap \calH) \bigg] \P(\calI \cap \calJ \cap \calH) \notag \\
    &= \P\bigg(\bigcap_{m=1}^\ell \calA_0^m \mid \calI \cap \calJ \cap \calH\bigg) \bigg[ \prod_{m=1}^\ell\prod_{i=1}^{n_m} \P(\calA_i^m \mid \calA_0^m, \ldots, \calA_{i-1}^m, \calI \cap \calJ \cap \calH) \bigg] \P(\calI \cap \calJ \cap \calH)
\end{align}
Now, we bound each factor of \eqref{eq:Ai_all_blocks_1d}. For the first factor, we apply Proposition \ref{prop:MAP_1d} to obtain that
\begin{equation}
\label{eq:Ai_first_block_1d}
    \P\bigg(\bigcap_{m=1}^\ell \calA_0^m \mid \calI \cap \calJ \cap \calH\bigg) = 1 - o(1).
\end{equation}
For the second factor, we condition each term on the configuration $C_{p(i)}^m$, and use an argument analogous to the one leading to \eqref{eq:Ai_remaining_blocks} to obtain
\begin{equation}
\label{eq:Ai_remaining_blocks_1d}
    \P(\calA_i^m \mid \calA_0^m, \ldots, \calA_{i-1}^m, \calI \cap \calJ \cap \calH) \geq 1 - c_4^{\log n} n^{-c_5 \log n}.
\end{equation}
For the third factor, we use Lemma \ref{lem:maximum_vertices_per_block_1d}, Lemma \ref{lem:propagate_distinguishing_vertices}, and Lemma \ref{lem:max_initial_blocks_1d} to obtain that
\begin{equation}
\label{eq:IJK_prob_1d}
    \P(\calI \cap \calJ \cap \calH) = 1 - o(1).
\end{equation}
Substituting \eqref{eq:Ai_first_block_1d}, \eqref{eq:Ai_remaining_blocks_1d}, and \eqref{eq:IJK_prob_1d} into \eqref{eq:Ai_all_blocks} and using Bernoulli's inequality yields
\begin{align*}
    \P\bigg(\bigcap_{m=1}^\ell \bigcap_{i=1}^{n_m} \calA_i^m\bigg) &\geq (1 - o(1))\Big(1 - c_4^{\log n} n^{-c_5 \log n}\Big)^{2n/(r \log n)} (1 - o(1)) \\
    &\geq (1 - o(1))\bigg(1 - \frac{2 c_4^{\log n} n^{1 - c_5 \log n}}{r \log n} \bigg) (1 - o(1)) \\
    &= 1 - o(1),
\end{align*}
which shows that $\hat{\sigma}$ makes at most $c \log n$ mistakes on all $\delta$-occupied blocks with high probability. Since $\delta$ and $c$ can be made arbitrarily small, we conclude that $\hat{\sigma}$ achieves almost exact recovery.
\end{proof}

Once we have an almost-exact estimator $\hat{\sigma}$, \texttt{Refine} produces an exact estimator $\tilde{\sigma}$ as shown in Section \ref{subsec:refine}. Thus, we achieve exact recovery when $d = 1$, $| \Omega_{\pi, P}| = 1$.

\section{Proof of Impossibility}
\label{appx:proof_of_impossibility}
In this section, we prove Theorem \ref{thm:impossibility}, our impossibility result. To show that exact recovery is impossible, we will show that the MAP estimator fails with a constant, nonzero probability, which implies that any estimator fails with probability bounded away from zero because the MAP is optimal. The proof of impossibility for $d=1$, $\lambda r < 1$, and $|\Omega_{\pi, P}| \geq 2$ was provided in Section \ref{sec:proof_sketch_impossibility}. Therefore, we just need to show that exact recovery is impossible when $\lambda \nu_d r^d \min_{i \neq j} D_+(\theta_i \Vert \theta_j; \pi, g) < 1$.

\begin{prop}
\label{prop:impossibility_proof_1}
    Any estimator $\tilde{\sigma}: V \to Z$ fails to achieve exact recovery for $G \sim \text{GHCM}(\lambda, n, r, \pi, P(y), d)$ when $\lambda \nu_d r^d \min_{i \neq j} D_+(\theta_i \Vert \theta_j; \pi, g) < 1$.
\end{prop}
We will split the proof into two cases: $\lambda \nu_d r^d < 1$ and $\lambda \nu_d r^d \geq 1$.

Case 1: $\lambda \nu_d r^d < 1$. In this case, we will show that the vertex visibility graph contains an isolated vertex with high probability. Consequently, there is a constant, nonzero probability that the MAP estimator fails since it fails to label the isolated vertex correctly with a constant, nonzero probability.

To accomplish this, we note that the vertex visibility graph is a \textit{random geometric graph}. We define the random geometric graph $RGG(d, n, x)$ as follows: First, we distribute vertices within the $d$-dimensional unit torus by a Poisson point process of intensity $n$. Then, for any pair of vertices, we form an edge if and only if the distance between the vertices is less than $x$. 

The vertex visibility graph of $G$ is an $RGG(d, \lambda n, r(\frac{\log n}{n})^{1/d})$ because points are distributed on the volume $n$ torus by a Poisson point process of intensity $\lambda$ with visibility radius $r(\log n)^{1/d}$. Rescaling the volume by $1/n$ yields a Poisson point process of intensity $\lambda n$ on a unit torus with visibility radius $r(\frac{\log n}{n})^{1/d}$.

Now, we use the following result from the literature on random geometric graphs.
\begin{lemma}[Theorem 7.1 in \cite{Penrose2003}]
\label{lem:isolated_vertex}
    Let $M_n := \inf_{x \in \R}\{\text{RGG}(d, n, x) \text{ has no isolated vertex}\}$. Then, 
    \begin{equation*}
        \lim_{n \to \infty} \frac{n \nu_d M_n^d}{\log n} = 1 \quad \text{a.s.},
    \end{equation*}
    which implies that
    \begin{equation*}
        M_n \sim \bigg(\frac{\log n}{n \nu_d}\bigg)^{1/d} \quad \text{a.s.}
    \end{equation*}
\end{lemma}
Lemma \ref{lem:isolated_vertex} considers a Poisson process of intensity $n$; for a Poisson process of intensity $\lambda n$, we obtain 
\begin{equation*}
    \lim_{n \to \infty} \frac{\lambda n \nu_d M_{\lambda n}^d}{\log \lambda n} = 1.
\end{equation*}
Thus, there exists an isolated vertex on the vertex visibility graph with high probability if the visibility radius is less than $M_{\lambda n}$. Since the visibility radius is $r(\frac{\log n}{n})^{1/d}$ and $M_{\lambda n} \sim (\frac{\log (\lambda n)}{\lambda n \nu_d})^{1/d}$, we obtain that there exists an isolated vertex with high probability if $\lambda \nu_d r^d < 1$. Consequently, the MAP estimator fails with a constant, non-zero probability when $\lambda \nu_d r^d < 1$. Therefore, any estimator fails to achieve exact recovery in Case 1.

Case 2: $\lambda \nu_d r^d \geq 1$. In this case, Lemma \ref{lem:impossibility_lemma} gives a sufficient condition for the MAP estimator to fail. As mentioned in Section \ref{sec:proof_sketch_impossibility}, we will adapt the proofs of Lemma B.4 and Lemma B.5 in \cite{Gaudio+2025}.

\begin{lemma}
\label{lem:impossibility_lemma_1}
    Let $G \sim \text{GHCM}(\lambda, n, r, \pi, P(y), d)$, and suppose that $\lambda \nu_d r^d \geq 1$ and $\lambda \nu_d r^d \min_{i \neq j} D_+(\theta_i \Vert \theta_j; \pi, g) < 1$. Then, \eqref{eq:impossibility_condition_1} is satisfied.
\end{lemma}
\begin{proof}
    It is sufficient to show that $\P(0 \text{ is Flip-Bad in } G \cup \{0\}) \geq n^{-1 + \beta}$. Observe that 0 is Flip-Bad if and only if the posterior likelihood of an incorrect label (with respect to the true labels of all other vertices) is larger than the posterior likelihood of the correct label. Therefore, letting $\sigma^*(0) = i$, we would like to show that there exists $j \neq i$ such that
    \begin{equation}
    \label{eq:0_is_FlipBad}
        \P\bigg( \log(\pi_j) + \sum_{v \in V, v \sim 0} \log p_{j, \sigma^*(v)}(x_{0, v}; \|v\|) \geq \log(\pi_i) + \sum_{v \in V, v \sim 0} \log p_{i, \sigma^*(v)}(x_{0, v}; \|v\|) \bigg) \geq n^{-1+\beta}.
    \end{equation}

    Fix $j$ such that $\lambda \nu_d r^d D_+(\theta_i \Vert \theta_j; \pi, g) < 1$. Rearranging $\eqref{eq:0_is_FlipBad}$, we would like to lower bound
    \begin{equation}
    \label{eq:0_is_FlipBad_rearranged}
        P_{\text{Flip-Bad}} := \P\bigg( \sum_{v \in V, v \sim 0} \log\Big(\frac{p_{j, \sigma^*(v)}(x_{0, v}; \|v\|)}{p_{i, \sigma^*(v)}(x_{0, v}; \|v\|)}\Big) \geq \log\Big(\frac{\pi_i}{\pi_j}\Big) \bigg).
    \end{equation}
    We now construct a random variable $X$ such that 
    \begin{equation*}
        X \sim \log\bigg(\frac{p_{j, \sigma^*(v)}(x_{0, v}; \|v\|)}{p_{i, \sigma^*(v)}(x_{0, v}; \|v\|)}\bigg)
    \end{equation*}
    where $v$ is a vertex placed uniformly at random in the neighborhood of 0 and assigned a community label $\sigma^*(v)$ according to the probability vector $\pi$. Let $D$ be a random variable with density
    \begin{equation*}
        f_D(x) :=
        \begin{cases}
            dx^{d-1}/r^d \text{ if } 0 \leq x \leq r \\
            0 \text{ otherwise.} 
        \end{cases}
    \end{equation*}
    which represents the distance between 0 and a vertex $v$ placed uniformly at random in the neighborhood of 0. Let $X$ be a mixture distribution of $X_s$ for $s \in Z$, with probability $\pi_s$, which represents the community of $v$. Let $X_{is}$ be a random variable coupled to $D$, such that conditioned on $D = y$, we sample $X_{is}$ from $P_{is}(y)$; note that $X_{is}$ represents the edge weight between $v$ to 0. Finally, let
    \begin{equation*}
        X_s := \log\bigg( \frac{p_{js}(X_{is} ;D)}{p_{is}(X_{is} ;D)} \bigg).
    \end{equation*}   
    as in the proof of Lemma \ref{lem:correct_labels_MAP_failures}. Letting $N(0) \sim \text{Pois}(\lambda \nu_d r^d \log n)$ and $\{X_\ell\}_{\ell \in \N} \overset{iid}{\sim} X$, from \eqref{eq:0_is_FlipBad_rearranged} we obtain that
    \begin{align}
    \label{eq:0_is_FlipBad_LOTP}
        P_{\text{Flip-Bad}} &= \P\bigg( \sum_{\ell = 1}^{N(0)} X_\ell \geq \log\Big(\frac{\pi_i}{\pi_j}\Big) \bigg) \notag \\
        &= \sum_{m = 0}^\infty \P(N(0) = m) \P\bigg( \sum_{\ell = 1}^m X_\ell \geq \log\Big(\frac{\pi_i}{\pi_j}\Big)\bigg)
    \end{align}
    Now, we lower-bound the probability of each term in \eqref{eq:0_is_FlipBad_LOTP}. Fix $\alpha > 0$ and $\delta > 0$, which we will determine later. Using Cram\'er's theorem (Lemma \ref{lem:Cramer}), we have that for any fixed $\epsilon > 0$, there exists $M$ such that 
    \begin{equation*}
    \P\bigg(\sum_{\ell = 1}^m X_\ell \geq m\alpha \bigg) \geq \exp\Big(m(-\Lambda^*_X(\alpha) - \epsilon)\Big) \quad \text{for all } m \geq M
    \end{equation*}
    For sufficiently large $n$, we have that $\alpha \delta \log n \geq \log(\pi_i/\pi_j)$. Thus, for any $m \geq \delta \log n$, we obtain that
    \begin{equation*}
        \P\bigg(\sum_{\ell = 1}^m X_\ell \geq \log\Big(\frac{\pi_i}{\pi_j}\Big)\bigg) \geq \P\bigg(\sum_{\ell = 1}^m X_\ell \geq m\alpha \bigg) \geq \exp\Big(m(-\Lambda^*_X(\alpha) - \epsilon)\Big)
    \end{equation*}
    Furthermore, we also have that $\delta \log n \geq M$ for sufficiently large $n$. Hence, we can write \eqref{eq:0_is_FlipBad_LOTP} as
    \begin{align}
    \label{eq:0_is_FlipBad_simplified} 
        P_\text{Flip-Bad} &= \sum_{m=0}^\infty \P\bigg( \sum_{\ell = 1}^m X_\ell \geq \log\Big(\frac{\pi_i}{\pi_j}\Big)\bigg) \P(N(0) = m) \notag \\
        &\geq \sum_{m=\delta \log n}^\infty \P\bigg( \sum_{\ell = 1}^m X_\ell \geq \log\Big(\frac{\pi_i}{\pi_j}\Big)\bigg) \P(N(0) = m) \notag \\
        &\geq \sum_{m=\delta \log n}^\infty \exp\Big(m(-\Lambda^*_X(\alpha) - \epsilon)\Big) \P(N(0) = m) \notag \\
        &= \sum_{m=0}^\infty \exp\Big(m(-\Lambda^*_X(\alpha) - \epsilon)\Big) \P(N(0) = m) \notag \\
        &\qquad- \sum_{m=0}^{\delta \log n - 1} \exp\Big(m(-\Lambda^*_X(\alpha) - \epsilon)\Big) \P(N(0) = m) \notag \\
        &\geq \bigg[ \sum_{m=0}^\infty \exp\Big(m(-\Lambda^*_X(\alpha) - \epsilon)\Big) \P(N(0) = m) \bigg] - \P(N(0) < \delta \log n) \notag \\
        &= \E\Big[\exp\Big(N(0)(-\Lambda^*_X(\alpha) - \epsilon) \Big)\Big] - \P(N(0) < \delta \log n) \notag \\
        &= n^{\lambda \nu_d r^d \big(\exp\big(-\Lambda^*_X(\alpha) - \epsilon\big) - 1\big)} - \P(N(0) < \delta \log n).
    \end{align}
    where last equality follows from the moment-generating function of the Poisson random variable $N(0)$. 

    Now, let $t_{ij}$ be defined as in \eqref{def:t_ij} and define $\Lambda(t) := \log \E[e^{tX}]$ with derivative $\Lambda'(t)$. In \cite[Lemma B.4]{Gaudio+2025}, the authors showed that for any $\eta > 0$, taking $0 < \alpha < \Lambda'(t_{ij} + \eta)$ yields that $\Lambda^*_X(\alpha) \leq (t_{ij} + \eta)\alpha - \Lambda(t_{ij})$. Combining this result with \eqref{eq:0_is_FlipBad_simplified}, we obtain that for any $\eta > 0$, there exists $\alpha > 0$ such that
    \begin{equation}
    \label{eq:0_is_FlipBad_with_eta}
        P_\text{Flip-Bad} \geq n^{\lambda \nu_d r^d \big(\exp\big(-(t_{ij} + \eta)\alpha + \Lambda(t_{ij}) - \epsilon\big) - 1\big)} - \P(N(0) < \delta \log n).
    \end{equation}
    Now, we show that there exist $\epsilon > 0$, $\eta > 0$, $\alpha > 0$, and $\beta > 0$ such that 
    \begin{equation}
    \label{eq:FlipBad_exponent}
        \lambda \nu_d r^d  \Big(\exp\big(-(t_{ij} + \eta)\alpha + \Lambda(t_{ij}) - \epsilon\big) - 1\Big) > -1 + 2\beta.
    \end{equation}
    By \eqref{def:t_ij} and \eqref{eq;CH_divergence_alternate}, we obtain that
    \begin{equation*}
    \begin{aligned}
        &\lambda \nu_d r^d  \Big(\exp\big(-(t_{ij} + \eta)\alpha + \Lambda(t_{ij}) - \epsilon\big) - 1\Big) \\
        &\quad= \lambda \nu_d r^d  \bigg(\exp\bigg(-\epsilon -(t_{ij} + \eta)\alpha + \log\Big(\sum_{a \in Z} \pi_a \overline{\phi}_{t_{ij}}(P_{ia}, P_{ja}) \Big) \bigg) - 1\bigg) \\
        &\quad= \lambda \nu_d r^d \bigg(\Big(\sum_{a \in Z} \pi_a \overline{\phi}_{t_{ij}}(P_{ia}, P_{ja}) \Big) \cdot \exp\big(-\epsilon - (t_{ij} + \eta) \alpha \big) - 1 \bigg) \\
        &\quad= \lambda \nu_d r^d \bigg( \big(1 - D_+(\theta_i, \theta_j)\big) \cdot \exp\big(-\epsilon - (t_{ij} + \eta) \alpha \big) - 1 \bigg)
    \end{aligned}
    \end{equation*}
    For any $\gamma > 0$, observe that there exist $\epsilon > 0$, $\eta > 0$, and $\alpha > 0$ sufficiently small such that $\exp(-\epsilon - (t_{ij}+\eta)\alpha) \geq 1 - \gamma$. Hence, 
    \begin{align}
    \label{eq:FlipBad_exponent_simplified}
        &\lambda \nu_d r^d  \Big(\exp\big(-(t_{ij} + \eta)\alpha + \Lambda(t_{ij}) - \epsilon\big) - 1\Big) \notag \\
        &\quad\geq \lambda \nu_d r^d \bigg( \big(1 - D_+(\theta_i \Vert \theta_j; \pi, g)\big) \cdot (1 - \gamma) - 1 \bigg) \notag \\
        &\quad= -\lambda \nu_d r^d \bigg( \big(D_+(\theta_i \Vert \theta_j; \pi, g)\big) \cdot (1 - \gamma) + \gamma \bigg)
    \end{align}
    Observe that \eqref{eq:FlipBad_exponent_simplified} $\to -\lambda \nu_d r^d D_+(\theta_i \Vert \theta_j; \pi, g)$  as $\gamma \to 0$. Hence, let $\beta := (1 - \lambda \nu_d r^d D_+(\theta_i \Vert \theta_j; \pi, g)) / 3$ and note that $\beta > 0$ because $\lambda \nu_d r^d D_+(\theta_i \Vert \theta_j; \pi, g) < 1$. Then, we obtain that for sufficiently small $\gamma$, we have
    \begin{equation*}
        -\lambda \nu_d r^d \bigg( \big(D_+(\theta_i \Vert \theta_j; \pi, g)\big) \cdot (1 - \gamma) + \gamma \bigg) > -1 + 2\beta.
    \end{equation*}
    Therefore, there exist $\epsilon > 0$, $\eta > 0$, $\alpha > 0$, and $\beta > 0$ such that \eqref{eq:FlipBad_exponent} holds. Taking these values of $\epsilon$, $\eta$, $\alpha$, and $\beta$, from \eqref{eq:0_is_FlipBad_with_eta} we obtain that 
    \begin{equation}
    \label{eq:0_is_FlipBad_first_term_bounded}
        P_\text{Flip-Bad} \geq n^{-1 + 2\beta} - \P(N(0) < \delta \log n).
    \end{equation}
    Finally, we show that for any fixed $0 < c < 2\beta$, there exists $\delta > 0$ such that $\P(N(0) < \delta \log n) < n^{-1 + c}$. Since $N(0) \sim \text{Pois}(\lambda \nu_d r^d \log n)$, the Chernoff bound (Lemma \ref{lem:poisson_Chernoff}) yields
    \begin{align*}
        \P(N(0) < \delta \log n) &\leq \exp\bigg(\delta \log n - \lambda \nu_d r^d \log n - \log\Big(\frac{\delta}{\lambda \nu_d r^d}\Big) \delta \log n\bigg) \\
        &= n^{\delta - \lambda \nu_d r^d - \delta \log(\delta/(\lambda \nu_d r^d))}.
    \end{align*}
    Let $f(\delta) := \delta - \lambda \nu_d r^d - \delta \log(\delta/(\lambda \nu_d r^d))$. Observe that $f(\delta) \to -\lambda \nu_d r^d$ as $\delta \to 0$ and also that $f(\delta)$ is strictly increasing on $0 < \delta < \lambda \nu_d r^d$. Therefore, we can let $\delta > 0$ sufficiently small such that $f(\delta) < -\lambda \nu_d r^d + c$. Furthermore, since $-\lambda \nu_d r^d \leq -1$ by hypothesis, we obtain that $f(\delta) < -1 + c$. As a result, we obtain that
    \begin{equation*}
        \P(N(0) < \delta \log n) \leq n^{-1 + c}
    \end{equation*}
    which combined with \eqref{eq:0_is_FlipBad_first_term_bounded} implies that
    \begin{equation*}
        \P(0 \text{ is Flip-Bad in } G \cup \{0\}) \geq n^{-1 + 2\beta} - n^{-1 + c} > n^{-1+\beta}.
    \end{equation*}
     Therefore, \eqref{eq:impossibility_condition_1} is satisfied.
\end{proof}

\begin{lemma}
\label{lem:impossibility_lemma_2}
    Let $G \sim \text{GHCM}(\lambda, n, r, \pi, P(y), d)$, and suppose that $\lambda \nu_d r^d \geq 1$ and $\lambda \nu_d r^d \min_{i \neq j} D_+(\theta_i \Vert \theta_j; \pi, g) < 1$. Then, \eqref{eq:impossibility_condition_2} is satisfied.
\end{lemma}
\begin{proof}
The proof is identical to the proof of \cite[Lemma 13]{GaudioJin2025}, which itself is an adaption of \cite[Lemma B.5]{Gaudio+2025}, apart from the different IT threshold. Concretely, we invoke Lemma \ref{lem:impossibility_lemma_1} to obtain $\E^0[\1_{0 \text{ is Flip-Bad in } G \cup \{0\}}] \geq n^{-1 + \beta}$ for some constant $\beta > 0$ when $\lambda \nu_d r^d \geq 1$ and $\lambda \nu_d r^d \min_{i \neq j} D_+(\theta_i \Vert \theta_j; \pi, g) < 1$.
\end{proof}

Combining Lemmas \ref{lem:impossibility_lemma_1} and \ref{lem:impossibility_lemma_2} with Lemma \ref{lem:impossibility_lemma}, we conclude that exact recovery is impossible in Case 2. This completes the proof of Proposition \ref{prop:impossibility_proof_1}.

\end{document}